\documentclass[a4paper]{article}
\usepackage{amssymb,amsthm,amsmath,hyperref}
\usepackage{color}
\usepackage{authblk}
\usepackage{graphicx}

\usepackage{indentfirst}

\numberwithin{equation}{section}
\DeclareMathOperator{\diver}{\mathrm{div}}
\newtheorem{thm}{Theorem}[section]
\newtheorem{lem}{Lemma}[section]

\newtheorem{cor}{Corollary}[section]
\newtheorem{prop}{Proposition}[section]

\newcommand{\beq}{\begin{eqnarray}}
\newcommand{\eeq}{\end{eqnarray}}
\newcommand{\beqno}{\begin{eqnarray*}}
\newcommand{\eeqno}{\end{eqnarray*}}
\newcommand{\be}{\begin{equation}}
\newcommand{\ee}{\end{equation}}
\theoremstyle{definition}

\newtheorem{remark}{Remark}[section]

\allowdisplaybreaks

\begin{document}
\title{\Large On the initial-boundary value problem for the 2D partially dissipative Oldroyd-B model: global well-posedness and large time stability}
\author[1]{\rm Zhenrong Nong}
\author[2]{\rm Yinghui Wang\footnote{Corresponding author.\\ $\quad\quad\quad$ E-mail addresses: yhwangmath@hunnu.edu.cn, yhwangmath@163.com (Y. Wang).}}
\author[3]{\rm Huancheng Yao}
\author[4]{\rm Shihao Zhang}
\affil[1,4]{\footnotesize School of Mathematics and Statistics, Hunan Normal University, Changsha, Hunan 410081, P. R. China}
\affil[2]{\footnotesize MOE-LCSM, School of Mathematics and Statistics, Hunan Normal University, Changsha, Hunan 410081, P. R. China}
\affil[3]{\footnotesize College of Mathematics and Informatics, South China Agricultural University, Guangzhou, Guangdong 510642, P. R. China}
\date{}
\maketitle
\begin{abstract}
This work studies the global well-posedness of the Oldroyd-B model with anisotropic viscosity. While global existence and uniqueness of strong solutions for the fully dissipative Oldroyd-B model were established in [Constantin-Kliegl, Arch. Ration. Mech. Anal., 206 (2012), 725--740], under $H^2(\mathbb{R}^2)$ initial data,
the horizontally viscous counterpart—where dissipation of velocity acts only along the horizontal direction—remains  unexplored.  
We establish the global existence and uniqueness of strong solutions to the initial-boundary value problem for the horizontally viscous Oldroyd-B model  with arbitrary large initial data in the vertical strip domain $[0,1]\times\mathbb{R}$ and the periodic channel $[0,1]\times\mathcal{T}$, where $\mathcal{T}$ represents the one-dimensional torus.    
To the best of our knowledge, this work provides the first rigorous proof of global-in-time existence and uniqueness of strong solutions for the partially dissipative Oldroyd-B model. 
In addition, we investigate the long-time asymptotic behavior of solutions for small initial data in the periodic channel $[0,1]\times\mathcal{T}$. 
Our results extend the current understanding of viscoelastic fluid dynamics under partial dissipation constraints.
\vspace{4mm}


\end{abstract}

\section{Introduction}
The Oldroyd-B model serves as a fundamental framework for characterizing viscoelastic fluid dynamics, as comprehensively discussed in \cite{Bird-Armstrong-Hassager19871,Bird-Armstrong-Hassager19872}.  We begin by the following generalized two-dimensional Oldroyd-B system, which describes the motion of dilute polymeric fluids (see \cite{Barrett_Lu_Suli_2017,Constantin-Kliegl2012,La_2020} for a detailed illustration of this model):
\begin{equation}\label{EQ-t2h} 
\begin{cases}
 \partial_{t}u+u\cdot\nabla u+\nabla P-\mu\partial^2_x u=\diver\mathbb{T}, \\
 \partial_{t}\eta+u\cdot\nabla\eta=\varepsilon\Delta\eta,\\
 \partial_{t}\mathbb{T}+u\cdot\nabla\mathbb{T}-(\nabla u\mathbb{T}+\mathbb{T}\nabla^\top u)=\varepsilon\Delta\mathbb{T}-2\kappa(\mathbb{T}-\eta\mathbb{I}), \\
 \nabla\cdot u=0, 
 \end{cases} 
\end{equation}
where $u$ denotes the fluid velocity, $\mathbb{T}$ is the symmetric extra stress tensor (a  $2\times 2$   matrix), $P$ is the  pressure, $\eta$ is the density of suspended polymer, $\mu$ denotes the viscosity coefficient, $\varepsilon$ is the center-of-mass diffusion coefficient, $\kappa$ is the relaxation parameter. And, $\varepsilon,~\mu, ~\kappa$ are positive constants. By introducing the modified stress tensor $\tau:=\mathbb{T}-\eta\mathbb{I}$ with $\mathbb{I}$ being a 2 $\times$ 2 unit matrix, one can derive an equivalent system of \eqref{EQ-t2h} as follows:
\begin{equation}\label{EQ-t1h} 
\begin{cases}
 \partial_{t}u+u\cdot\nabla u+\nabla P-\mu\partial^2_x u=\diver \tau, \\
 \partial_{t}\eta+u\cdot\nabla\eta=\varepsilon\Delta\eta,\\
 \partial_{t}\tau+u\cdot\nabla\tau-(\nabla u\tau+\tau\nabla^\top u)+2\kappa\tau-\varepsilon\Delta\tau=2\eta\mathbb{D}(u), \\
 \nabla\cdot u=0, 
 \end{cases} 
\end{equation}
where $\mathbb{D}(u):=(\nabla u+\nabla^\top u)/2$ is the symmetric part of velocity gradient. In the present paper, we will establish  the global well-posedness for the initial-boundary problem of the Oldroyd-B model \eqref{EQ-t1h} over two distinct domains:  
 $[0,1]\times\mathbb{R}$ and $[0,1]\times\mathcal{T}$, where $\mathcal{T}$ represents the one-dimensional periodic torus. More specifically,  we equip system \eqref{EQ-t1h} with the initial value conditions: 
     \begin{gather} 
u(x,y,0)=u_0(x,y),~\tau(x,y,0)=\tau_0(x,y),~\eta(x,y,0)=\eta_0(x,y),
    \end{gather}
    for $(x,y) \in [0,1]\times\mathbb{R}$ or $[0,1]\times\mathcal{T}$;
and the boundary value conditions:
    \begin{gather}
                u(0,y,t) = u(1,y,t),~\partial_{x}\eta(0,y,t) = \partial_{x}\eta(1,y,t) = 0,~\partial_{x}\tau(0,y,t) = \partial_{x}\tau(1,y,t) = 0, \label{OB_as_boundary_1}
    \end{gather}
     for $(y,t) \in \mathbb{R}\times \mathbb{R}_+$ or $\mathcal{T}\times \mathbb{R}_+ $.
When analyzing the problem in periodic channel $[0,1]\times\mathcal{T}$, system \eqref{EQ-t1h} is also supplemented with the periodic boundary conditions in  $y$-direction, in addition to  the boundary conditions \eqref{OB_as_boundary_1}. 

We remark that the diffusive term 
 $-\varepsilon\Delta\tau$ in \eqref{EQ-t1h}$_3$ is introduced to account for the shear and vorticity banding phenomena observed in various studies, cf. \cite{Bhave2,Cates_2006,Dhont_2008,El-Kareh-Leal1989, Malek_etal_2018}. However, the diffusion coefficient  $\varepsilon$ is significantly smaller in magnitude compared to other relevant parameters (cf. \cite{Bhave1}). Hence, in the early works on Oldroyd-B type models, the diffusive term $ -\varepsilon \Delta \tau$  is usually neglected. In consideration of mathematics, the case $\varepsilon>0$ (the diffusive model) and the case $\varepsilon = 0$ (the non-diffusive model) are significantly different. More specifically, the present of diffusive term $ -\varepsilon \Delta \tau$ in \eqref{EQ-t1h}$_3$ makes the equation of $\tau$ become parabolic PDEs, which may cause more stability than the hyperbolic one (i.e., the case $\varepsilon = 0$). See \cite{Barrett_Boyaval_2011,Barrett_Lu_Suli_2017,Constantin-Kliegl2012,Elgindi_Liu_2015,Elgindi_Rousset_2015,Huang_Wang_Wen_Zi_2022,Liu_Wang_Wen_2023,Wang_Wen_2024} for the mathematical analysis of the diffusive Oldroyd-B models. Moreover, the only horizontally viscous setting in \eqref{EQ-t1h}$_3$ is due to the studies of a lot of real fluids such as turbulent flows in
Ekman layers (cf. \cite{Pedlosky_1987}). There are  extensive  works on the mathematical analysis of horizontally viscous Navier-Stokes equations, magnetohydrodynamics equations and Boussinesq equations, see the recent advances in \cite{Cheng_Ji_Tian_Wu_2025,Wu_Zhu_2021} and the reference therein for comprehensive discussions.  In contrast, the horizontally viscous Oldroyd-B model remains significantly understudied. To the authors' best knowledge, only \cite{Feng-Wang-Wu2023} has achieved progress in this direction, establishing the existence of small strong solutions via sophisticated energy methods. Crucially, no rigorous results currently exist for the well-posedness of this system with arbitrarily large initial data. The main goal of this article is to give a positive answer to this problem. \\[0.5mm]

 Let us give a brief overview of some relevant works on the related models.\\[1mm]
 \noindent{\bfseries The non-diffusive model}
 
 In  \cite{Oldroyd_1950},
 the following  non-diffusive model was introduced by  Oldroyd to describe the motion of incompressible viscous fluids with significant elastic effects:
     \begin{equation}\label{OB_incom_classical}
        \begin{cases}
        u_t+u\cdot\nabla u +\nabla P -\mu \Delta u
        =\mathrm{div}  \tau,\\
       \tau_t+u\cdot \nabla \tau-(\nabla u\tau+\tau\nabla^\top u)+ \gamma \tau
         =  k  \mathbb{D}u,\\
        \mathrm{div} u = 0.
        \end{cases}
    \end{equation}
    Indeed, \eqref{OB_incom_classical}  represents the full viscous version of \eqref{EQ-t1h}.
    Some significant research progress has been made for the non-diffusive model (\ref{OB_incom_classical}) in recent years. For the Dirichlet problem in two-dimensional and three-dimensional bounded domain, 
Guillop\'{e} and Saut in \cite{Guillope-Saut1990}  established the local and global well-posedness for small initial data in the $H^s(s>2)$ framework. Also see the works of
 Fern\'{a}ndez-Cara,  Guill\'{e}n and  Ortega in \cite{ Fernandez-Guillen-Ortega1998} for the results in  $W^{1,p}(p>3)$, and Chemin and Masmoudi in \cite{Chemin-Masmoudi2001} for the results in critical Besov spaces. The results in \cite{Guillope-Saut1990} and \cite{Chemin-Masmoudi2001} were further advanced by Molinet and Talhouk in \cite{Molinet-Talhouk2004} and Zi, Fang and Zhang in \cite{Zi-Fang-Zhang2014} without the smallness restriction on the coupling constant (i.e., $k$ in \eqref{OB_incom_classical}). 
For the three-dimensional exterior domains problem, Hieber, Naito and Shibata in \cite{Hieber-Naito-Shibata2012} obtained the global well-posedness of small strong solutions with a small coupling constant. Fang, Hieber and Zi in \cite{Fang-Hieber-Zi2013} later generalized this result to handle arbitrary coupling constants. For weak solutions,  Lions and Masmoudi in \cite{Lions_Masmoudi_2000} proved global existence for weak solutions of two-dimensional and three-dimensional corotational Oldroyd-B model, but the uniqueness remains open in 3-D. See the works by Chemin and Masmoudi in \cite{Chemin-Masmoudi2001}, Lei, Masmoudi and Zhou in \cite{Lei_Masmoudi_Zhou_2010}, Sun and Zhang in \cite{Sun-Zhang2011}, for the blow-up criteria of strong solutions. When considering the long-time behavior of solutions, the works of Hieber, Wen and Zi in \cite{Hieber-Wen-Zi2019}, Huang, Wang, Wen and Zi in \cite{Huang_Wang_Wen_Zi_2022} established the optimal polynomial decays of the strong solutions for   Cauchy problem in $\mathbb{R}^3$. The 2-D case was proved by Chen and Zhu in \cite{Chen-Zhu2023}. The vanishing viscosity limit was
 studied by Zi in \cite{Zi_2021} for 3-D Cauchy problem. One can also refer to \cite{Lei_2006,Liu_Lu_Wen_2021} for the works on the compressible counterpart of \eqref{OB_incom_classical}.\\[1mm]
\noindent{\bfseries The diffusive model ($\varepsilon>0$ and $-\mu\partial_{xx}$ is replaced by $-\mu\Delta$ in  \eqref{EQ-t1h})}
      
 The existence of global-in-time weak solutions was obtained by Barrett and Boyaval in \cite{Barrett_Boyaval_2011} for the 2-D case with positive constant $\eta$. For the 3-D case and some generalization, see the recent work by Bathory,  Bul\'{i}\v{c}ek and  M\'{a}lek in \cite{Bathory-Bulicek-Malek2021}.  Constantin and Kliegl in \cite{Constantin-Kliegl2012} established global existence and uniqueness of strong solution.  For the inviscid ($\mu = 0$) and constant $\eta$ case,  Elgindi  obtained the global well-posedness of strong  solution provided that the initial data are small enough. Later on, Elgindi and Liu in \cite{Elgindi_Liu_2015} extended the result in \cite{Elgindi_Rousset_2015} to the three-dimensional case. Huang, Wang, Wen and Zi in \cite{Huang_Wang_Wen_Zi_2022} and Wang in \cite{Wang_2023} established the optimal polynomial time decay rates of the strong solutions for Cauchy problem in $\mathbb{R}^3$. Recently, the relation of the diffusive model and non-diffusive model was studied by Wang and Wen in \cite{Wang_Wen_2024} for the 2-D initial-boundary  value problem with boundary layers. See also the works   \cite{Constantin-Wu-Zhao-Zhu2021,Dai-Tan-Wu2023,Wang-Wu-Xu-Zhong2022} and the reference therein for fractional dissipation problem.

\subsection{Notations} 
In the present paper, some standard notations are used.

(1) $C$ denotes a  generic positive constant which may depend on $\varepsilon, ~\kappa,~\mu,$ but is independent of temporal variable $t$.

(2) For scalar function $\eta$, vector function $u$ and matrix-valued function $\tau$,  we use the notations that 
$$(\nabla\eta)_{i}=\partial_{i}\eta,\; (\nabla u)_{ij}=\partial_{j}u_{i},\; \diver u=\partial_{i}u_{i},\; (\tau)_{ij}=\tau_{ij},\; (\nabla\tau)_{nml}=\partial_{l}\tau_{nm}.$$

(3) Let $X$ be a Banach space. Then, for constant $a,\;b,\;p$ with $-\infty \leq a< b \leq \infty$ and $1\leq p\leq\infty$,  $L^{p}(a,b;X)$ denotes the $L^{p}$-Lebesgue integrable function space from $[a,b]$ to $X$. Then its norm can be defined as
\begin{equation*} 
\|\nu\|_{L^{p}(a,b;X)}= \begin{cases} (\int^b_a{\|\nu\|^{p}_{X}}\mathrm{d}x)^{\frac{1}{p}},&~~\mathrm{for}~~1\leq p<\infty,\\[2mm]
 \text{ess}\sup_{t\in[a,b]}\|\nu\|_{X},&~~\mathrm{for}~~p=\infty.\\
\end{cases} 
\end{equation*}

\subsection{Main results}
The first result of this paper is  the global in time well-posedness for the initial-boundary value problem \eqref{EQ-t1h}--\eqref{OB_as_boundary_1}.  Let $\Omega$ be either $[0,1]\times\mathbb{R}$ or $[0,1]\times\mathcal{T}$, we have the following result.

\begin{thm}\label{thm4}
Assume that $u_0\in H^2(\Omega)$, $\eta_0$, $\tau_0 = \mathbb{T}_0 - \eta_0\mathbb{I}\in L^{1}(\Omega)\cap H^2(\Omega)$ with  $\diver u=0$, $\eta_0$ non-negative, and $\mathbb{T}_0$ non-negative definite. For arbitrary given $T>0$,   problem \eqref{EQ-t1h}--\eqref{OB_as_boundary_1} admits a unique strong solution satisfying 
\begin{equation*}
\begin{aligned}
&u\in L^\infty(0,T;H^2(\Omega)),\;\;\partial_t u\in L^\infty(0,T;L^2(\Omega))\cap L^2(0,T;H^1(\Omega)), \\
&\partial_x u\in L^2(0,T;H^2(\Omega)), \;\; \eta,\tau \in L^\infty(0,T;L^1(\Omega)\cap H^2(\Omega))\cap L^2(0,T;H^3(\Omega)),\\ 
&\partial_t\eta,\partial_t\tau\in L^\infty (0,T;L^2(\Omega))\cap L^2(0,T;H^1(\Omega)).
\end{aligned}
\end{equation*} 
\end{thm} 
\begin{remark}
By synthesizing the analytical framework from \cite{Constantin-Kliegl2012} with the methodology developed herein, we may extend Theorem \ref{thm4} to the case that the diffusion term $\varepsilon\Delta \eta$ in \eqref{EQ-t2h}$_2$ is omitted.
  Here, the retention of this diffusion term in  \eqref{EQ-t2h}$_2$ is  motivated by physical considerations  (as outlined in \cite{Barrett_Boyaval_2011}).
\end{remark}
\begin{remark}
     The foundational work of Constantin and Kliegl \cite{Constantin-Kliegl2012} established the global well-posedness of strong solutions for the fully viscous Oldroyd-B model of  two-dimensional Cauchy problem over $\mathbb{R}^2$. In this work, we focus on a physically relevant initial-boundary value problem. Specifically,  Theorem \ref{thm4} extends the results of  \cite{Constantin-Kliegl2012} to the initial boundary value problem and to the problem with only horizontal viscosity. 
\end{remark}
The long-time behavior of the solution $(u,\tau,\eta)$ obtained in Theorem \ref{thm4} is a very interesting problem.
The recent work \cite{Dong-Wu-Xu2021} reveal that even in the degenerate case $\tau = 0$ and $\eta = 0$ (corresponding to  the Navier-Stokes equation with only  horizontal viscosity),  rigorous results on long-time asymptotic stability has only been established under two fundamental requirements: (i) periodic strip domain $\mathcal{T}\times \mathbb{R}$, (ii) sufficiently small initial perturbations. 

Next, to study the long-time stability of the solution, we focus  on the case that $\Omega = [0,1]\times\mathcal{T}$.\footnote{The non-integrability of  $1$ over $[0,1]\times \mathbb{R}$  necessitates restricting our long-time stability analysis to compact domains. Our study therefore focuses on the case that  $\Omega = [0,1]\times\mathcal{T}$.}
Without loss of generality, assume  that 
$$
\int_{[0,1]\times\mathcal{T}}{\eta_0}\mathrm{d}x\mathrm{d}y=1.
$$
Then, we will  characterize the long-time asymptotic behavior of solutions around $\displaystyle (0,0, 1)$.  Letting $\bar{\eta} := \eta - 1$,
    then, the problem \eqref{EQ-t1h}--\eqref{OB_as_boundary_1} are equivalent to 
  
\begin{equation}\label{EQ-t4h} 
\begin{cases}
 \partial_{t}u+u\cdot\nabla u+\nabla P-\mu\partial^2_x u=\mathrm{div} \tau, \\
 \partial_{t}\bar{\eta}+u\cdot\nabla\bar{\eta}=\varepsilon\Delta\bar{\eta},\\
 \partial_{t}\tau+u\cdot\nabla\tau-(\nabla u\tau+\tau\nabla^\top u)+2\kappa\tau-\varepsilon\Delta\tau=2\bar{\eta}\mathbb{D}(u)+2\mathbb{D}(u), \\
 \nabla\cdot u=0,\\ 
 (u,\tau,\bar{\eta})|_{t=0} = (u_0,\tau_0,\eta_0-1),
 \end{cases} 
\end{equation} 
with the  boundary conditions that \begin{equation}\label{boundary}
\begin{aligned}
u(0,y,t)=u(1,y,t)=0,~\partial_x\tau(0,y,t)=\partial_x\tau(1,y,t)=0,~\partial_x\bar{\eta}(0,y,t)=\partial_x\bar{\eta}(1,y,t)=0,
\end{aligned}
\end{equation}
over $[0,1]\times\mathcal{T}$. For the solution of \eqref{EQ-t4h}--\eqref{boundary}, we have the following {\it a priori} estimates.
\begin{thm}\label{thm5}
Under the assumptions of Theorem \ref{thm4}, let $(u,\tau,\bar{\eta})$ be the solution of  \eqref{EQ-t4h}--\eqref{boundary} over $[0,T]$. Then, there exists a sufficiently small constant $\delta>0$ such that  the solutions of \eqref{EQ-t4h}--\eqref{boundary} satisfying
\begin{align*}
&\|(u,\bar{\eta},\tau)\|^{2}_{H^{1}}+\int^t_0{\|(\bar{\eta},\tau)\|^{2}_{H^{2}}}\mathrm{d}s
+\int^t_0{\|\partial_t u\|^{2}_{L^{2}}}\mathrm{d}s
+\int^t_0{\|(u,\partial_x u,\partial_x \partial_y u)\|^{2}_{L^{2}}}\mathrm{d}s \\
\leq\,& C\|(u_0,\bar{\eta}_0,\tau_0)\|^{2}_{H^{1}},
\end{align*}
for any $t\in [0,T]$, provided that
$$\|(u_0,\bar{\eta}_0,\tau_0)\|_{H^{1}([0,1]\times\mathcal{T})}\leq\delta.$$
\end{thm}
Then, we have the following characterization for  the long-time  behavior of $(u,\bar{\eta},\tau)$. 
\begin{thm}\label{thm6}
Under the assumption of Theorem \ref{thm5}, there exist two positive constants $C$ and $\nu$ such that
\begin{equation*}
\begin{aligned}
&\|(u,\bar{\eta},\tau)\|_{H^1}\leq C\|(u_0,\bar{\eta},\tau_0)\|_{H^1}\mathrm{e}^{-\nu t},~~~\forall t\in (0,+\infty),
\end{aligned}
\end{equation*}
where $C$ and $\nu$ are independent of $t$.
\end{thm}
\begin{remark}
    The long-time stability of problem  \eqref{EQ-t4h} over the spatial domains $\mathbb{R}^2$ or $[0,1]\times\mathbb{R}$ remains open.
\end{remark}
\subsection{Main ideas}
In order to establish the global well-posedness of the system \eqref{EQ-t1h}--\eqref{OB_as_boundary_1} for arbitrarily large initial data, the primary challenges stem from
two obstructions: (1) the strongly nonlinear stress coupling term $\nabla u\tau+\tau\nabla^\top u$, and (2) the absence of vertical dissipation in the velocity equation, precluding standard energy arguments. To overcome these difficulties, we synthesize the analytical frameworks developed by Constantin and Kliegl for the fully dissipative Oldroyd-B model in \cite{Constantin-Kliegl2012} and by Cao and Wu for partially dissipative magnetohydrodynamics in \cite{Cao-Wu2011}. The synthesis hinges on two critical observations: (1) the parabolic structure governing the evolution of $\eta$ and $\mathbb{T}$ ensures their non-negativity preservation through maximum principles. This structural property plays a pivotal role in initiating the bootstrap argument by providing essential control on stress configurations. (2)
the anisotropic Sobolev inequalities and the coupling structure which will  compensate for missing vertical regularization.

 For the large time stability analysis, we restrict on the spatial domain $[0,1]\times \mathcal{T}$. 
 The key point here is the Poincar\'{e} inequality over $[0,1]\times \mathcal{T}$, which will help us to complete the energy dissipation. The interplay between geometry property and the strong stress-velocity coupling structure forms the cornerstone of our asymptotic stability proof, ultimately demonstrating the exponential stability relaxation to equilibrium states.

The rest of the paper is organized as follows.
\begin{itemize}
\item In Section \ref{sec_p}, we present some preliminary results which will be used to prove the main Theorems.
\item In Section \ref{part2}, we will establish some {\it a priori} estimates, and prove the existence and uniqueness of  strong solution to problem \eqref{EQ-t1h}--\eqref{OB_as_boundary_1} via energy method.
\item In Section \ref{part3}, we consider the perturbed form of \eqref{EQ-t1h}--\eqref{OB_as_boundary_1}. We firstly establish a uniform $H^1$-bound for the solutions with small initial data. Then, we prove Theorem \ref{thm5}. Finally, we finish the proof of Theorem \ref{thm6} by rigorously characterizing long-time asymptotic behavior of the solutions.
\end{itemize}

\section{Preliminaries}\label{sec_p}
In this section, we introduce some useful results which will be used later.
To begin with, we have the following anisotropic inequalities.
\begin{lem}[Lemma 2.3 in \cite{Dong-Wu-Xu2021}]\label{2.3}
If a function $f(x,y)$ on $[0,1]\times\mathbb{R}$ satisfies $f\in H^{2}([0,1]\times\mathbb{R})$, then
\begin{equation*}
\begin{aligned}
\|f\|_{L^{\infty}}\leq&C\|f\|^\frac{1}{4}_{L^{2}}(\|f\|_{L^{2}}
+\|\partial_xf\|_{L^{2}})^\frac{1}{4}\|\partial_yf\|^\frac{1}{4}_{L^{2}}
(\|\partial_yf\|_{L^{2}}+\|\partial_{x}\partial_{y}f\|_{L^{2}})^\frac{1}{4}.
\end{aligned}
\end{equation*}
In addition, if $g\in H^{1}([0,1]\times\mathbb{R})$ and $h\in L^{2}([0,1]\times\mathbb{R})$, then the integral of the triple product over $[0,1]\times\mathbb{R}$ is bounded by
\begin{equation*}
\begin{aligned}
\int{|fgh|}\mathrm{d}x\mathrm{d}y\leq&C\|h\|_{L^{2}}\|g\|^{\frac{1}{2}}_{L^{2}}
\|\partial_{y}g\|^{\frac{1}{2}}_{L^{2}}\|f\|^{\frac{1}{2}}_{L^{2}}(\|f\|_{L^{2}}+\|\partial_{x}f\|_{L^{2}})^{\frac{1}{2}}.\\
\end{aligned}
\end{equation*}
\end{lem}
\begin{remark}\label{remark1}
In \cite{Dong-Wu-Xu2021}, the authors proved the results of  Lemma \ref{2.3} for the spatial domain $\mathcal{T}\times\mathbb{R}$, the results for the spatial domain $[0,1]\times\mathbb{R}$ in Lemma \ref{2.3} can be proved in a similar way.
Moreover, if the domain $[0,1]\times\mathbb{R}$ in Lemma \ref{2.3} is replaced by $[0,1]\times\mathcal{T}$, the corresponding results will be modified as follows:
\begin{align*}
&\|f\|_{L^{\infty}}\leq C\|f\|^\frac{1}{4}_{L^{2}}(\|f\|_{L^{2}}
+\|\partial_xf\|_{L^{2}})^\frac{1}{4}\|\partial_yf\|^\frac{1}{4}_{L^{2}}
(\|\partial_yf\|_{L^{2}}+\|\partial_{x}\partial_{y}f\|_{L^{2}})^\frac{1}{4}\\
&\quad\quad\quad\quad+C\|f\|^\frac{1}{2}_{L^{2}}\|\partial_x f\|^\frac{1}{2}_{L^{2}}
+C\|f\|_{L^{2}},\\
&\int{|fgh|}\mathrm{d}x\mathrm{d}y\leq C\|h\|_{L^{2}}\|g\|^{\frac{1}{2}}_{L^{2}}
(\|g\|_{L^{2}}+\|\partial_{y}g\|_{L^{2}})^{\frac{1}{2}}
\|f\|^{\frac{1}{2}}_{L^{2}}(\|f\|_{L^{2}}+\|\partial_{x}f\|_{L^{2}})^{\frac{1}{2}}.
\end{align*}
\end{remark}
\begin{prop}[Local well-posedness]\label{prop_local}
    Let $u_0\in H^2(\Omega)$ be divergence-free, $\eta_0 \in L^{1}(\Omega)\cap H^2(\Omega)$, $\tau_0\in L^{1}(\Omega)\cap H^2(\Omega)$  be a symmetric matrix. Then, there exists a positive time $T_0$ depending on the norms of $u_0,\eta_0,\tau_0$ such that problem  \eqref{EQ-t1h}--\eqref{OB_as_boundary_1} admits a unique strong solution satisfying 
\begin{equation*}
\begin{aligned}
&u\in L^\infty(0,T_0;H^2(\Omega)),~\partial_t u\in L^\infty(0,T_0;L^2(\Omega))\cap L^2(0,T_0;H^1(\Omega)), \\
&\partial_x u\in L^2(0,T_0;H^2(\Omega)),\; \eta,\tau\in L^\infty(0,T_0;L^1(\Omega)\cap H^2(\Omega))\cap L^2(0,T_0;H^3(\Omega)),\\ 
&(\partial_t\eta,\partial_t\tau)\in L^\infty (0,T_0;L^2(\Omega))\cap L^2(0,T_0;H^1(\Omega)).
\end{aligned}
\end{equation*} 
\end{prop}
\begin{remark}
    The prove of Proposition \ref{prop_local} is standard. In \cite{Constantin-Kliegl2012}, the authors proved the local well-posedness of a similar Cauchy problem. After some 
modification, one can easily prove the results in  Proposition \ref{prop_local}. There are also many other methods to prove the local well-posedness, one can refer to \cite{Bedrossian_Vicol_2022} and the reference therein for instance.
\end{remark}
To prove the global well-posedness, similar to the discuss in \cite{Constantin-Kliegl2012}, we need $\eta$ to be nonnegative and $\mathbb{T}= \tau +\eta \mathbb{I}$ to be nonnegative  definite. See also \cite{Barrett_Lu_Suli_2017,Bathory-Bulicek-Malek2021,Malek_etal_2018} for some related discuss on the nonnegativity of $\eta$ and $\mathbb{T}= \tau +\eta \mathbb{I}$. Specifically, we have the following Lemma.
\begin{lem}\label{3.1}
Let $(u,\eta,\tau)$ be strong solutions of problem \eqref{EQ-t1h}--\eqref{OB_as_boundary_1} with $u\in L^{1}(0,T;W^{1,\infty}(\Omega))$   divergence-free. Denote $\mathbb{T}:= \tau +\eta\mathbb{I}.$ Then $\eta$, $\mathrm{det}\mathbb{T}$ will stay nonnegative provided that $\eta_0$, $\mathrm{det}(\mathbb{T}_0)$ are nonnegative.
\end{lem}
\begin{proof}
Inspired by the work of \cite{Constantin-Kliegl2012}, let
\begin{equation*} 
\begin{cases}
 \Gamma_{1}=\frac{1}{2}(\mathbb{T}_{11}-\mathbb{T}_{22}), \\
 \Gamma_{2}=\mathbb{T}_{12}=\mathbb{T}_{21}, \\
 \Gamma_{3}=\mathrm{tr}\mathbb{T} = \mathbb{T}_{11}+\mathbb{T}_{22},
 \end{cases} 
 ~\text{and}~ 
\begin{cases}
 \alpha=\frac{1}{2}(\partial_xu_1-\partial_yu_2), \\
 \beta=\frac{1}{2}(\partial_xu_2+\partial_yu_1), \\
 \omega=\partial_xu_2-\partial_yu_1.
 \end{cases} 
\end{equation*}
Then, from \eqref{EQ-t1h}--\eqref{OB_as_boundary_1}, we have
\begin{equation}\label{Gamma} 
\begin{cases}
 \partial_t\eta+u\cdot\nabla\eta-\varepsilon\Delta\eta=0,\\
 \partial_{t}\Gamma_{1}+u\cdot\nabla\Gamma_{1}+2\kappa\Gamma_{1}-\varepsilon\Delta\Gamma_{1}=-\Gamma_{2}\omega+\alpha\Gamma_{3}, \\
 \partial_{t}\Gamma_{2}+u\cdot\nabla\Gamma_{2}+2\kappa\Gamma_{2}-\varepsilon\Delta\Gamma_{2}=\Gamma_{1}\omega+\Gamma_{3}\beta, \\
 \partial_{t}\Gamma_{3}+u\cdot\nabla\Gamma_{3}+2\kappa\Gamma_{3}-\varepsilon\Delta\Gamma_{3}=4\Gamma_{1}\alpha
+4\Gamma_{2}\beta+2\kappa\eta,
 \end{cases} 
\end{equation}
with the initial-boundary conditions that 
   \begin{gather*} 
      \eta(x,y,0)=\eta_0(x,y)\geq 0,~~\mathrm{det}(\mathbb{T}_0)=\left(\frac{\Gamma_{3}^2}{4}-(\Gamma^{2}_{1}+\Gamma^{2}_{2})\right)(x,y,0)\geq0, \\
\big(\partial_{x}\eta,\partial_{x}\Gamma_1,\partial_{x}\Gamma_2,\partial_{x}\Gamma_3\big)(0,y,t) = \big(\partial_{x}\eta,\partial_{x}\Gamma_1,\partial_{x}\Gamma_2,\partial_{x}\Gamma_3\big)(1,y,t) = 0.  
    \end{gather*}
Next, we prove the nonnegativity of $\eta$. Let $\lambda=-\eta$ in $\eqref{Gamma}_1$, then we have
\begin{equation}\label{lambda}
\begin{cases}
\partial_t\lambda(x,y,t)+u\cdot\nabla\lambda(x,y,t)-\varepsilon\Delta\lambda(x,y,t)=0,\\[1mm]
\eta(x,y,0)=\eta_0(x,y)\geq 0.
\end{cases}
\end{equation}
Multiplying \eqref{lambda}$_1$ by $\lambda_+:=\max\{0,\lambda\}$ and integrating the result over $\Omega$ by parts, we get
\begin{equation*}
\begin{aligned}
\frac{1}{2}\frac{\mathrm{d}}{\mathrm{d}t}\|\lambda_+(x,y,t)\|^{2}_{L^{2}}+\varepsilon\|\nabla\lambda_+(x,y,t)\|^{2}_{L^{2}}=0.
\end{aligned}
\end{equation*}
Integrating the above result over $(0,t)$, we get
\begin{equation*}
\begin{aligned}
\|\lambda_+(x,y,t)\|^{2}_{L^{2}}\leq\|\lambda_+(x,y,0)\|^{2}_{L^{2}}=0,~~\Rightarrow~~\lambda_+(x,y,t)\overset{a.e.}{=}0,
\end{aligned}
\end{equation*}
where the fact $\eta_0\geq0$ is used. Therefore, $$\eta(x,y,t)=-\lambda(x,y,t)\geq0.$$

Next, let
\begin{equation*}
\begin{aligned}
\gamma :=\Gamma_{3}-2\sqrt{\Gamma^{2}_{1}+\Gamma^{2}_{2}}.
\end{aligned}
\end{equation*}
A direct calculation shows that the nonnegativity of $\gamma$ is equivalent to  the nonnegativity of $\mathrm{det}\mathbb{T}$.  Multiplying \eqref{Gamma}$_2$ and \eqref{Gamma}$_3$ by $\Gamma_1$ and $\Gamma_2$ respectively, suming the result up, we get 
$$
\frac{1}{2}\big(\partial_t + u\cdot \nabla + 2\kappa - \varepsilon\Delta\big)(\Gamma_1^2 + \Gamma_2^2)  = -\varepsilon(|\nabla \Gamma_1|^2 + |\nabla \Gamma_2|^2)
+ (\alpha\Gamma_1 + \beta \Gamma_2)\Gamma_3,
$$
which implies, for any given $\theta>0$, that 
\begin{align}\label{eq_gamma12} 
          &\big(\partial_t + u\cdot \nabla + 2\kappa- \varepsilon\Delta\big)\sqrt{\Gamma_1^2 + \Gamma_2^2 + \theta^2}   \notag \\ 
          =\,& \varepsilon\frac{1}{\sqrt{\Gamma_1^2 + \Gamma_2^2 + \theta^2}}\left(\frac{|\Gamma_1\nabla \Gamma_1 + \Gamma_2\nabla \Gamma_2|^2}{\Gamma_1^2 + \Gamma_2^2 + \theta^2} - |\nabla \Gamma_1|^2 - |\nabla \Gamma_2|^2\right) \notag\\ 
         & + \frac{1}{\sqrt{\Gamma_1^2 + \Gamma_2^2 + \theta^2}}\Big((\alpha\Gamma_1 + \beta\Gamma_2)\Gamma_3 + 2\kappa\theta^2\Big) \notag\\
         =\,& -\varepsilon\frac{1}{ (\Gamma_1^2 + \Gamma_2^2 + \theta^2)^{\frac{3}{2}}}\Big(|\Gamma_1\nabla \Gamma_2 + \Gamma_2\nabla \Gamma_1|^2  + \theta^2(|\nabla\Gamma_1|^2 + |\nabla\Gamma_2|^2) \Big) \notag\\
         & + \frac{1}{\sqrt{\Gamma_1^2 + \Gamma_2^2 + \theta^2}}\Big((\alpha\Gamma_1 + \beta\Gamma_2)\Gamma_3 + 2\kappa\theta^2\Big) \notag\\ 
         \leq\,& \frac{1}{\sqrt{\Gamma_1^2 + \Gamma_2^2 + \theta^2}}\Big((\alpha\Gamma_1 + \beta\Gamma_2)\Gamma_3 + 2\kappa\theta^2\Big).
\end{align} 
Combining \eqref{eq_gamma12} and \eqref{Gamma}$_{4}$, we get 
\begin{align}
    &\left(\partial_t + u\cdot \nabla - \varepsilon\Delta + 2\kappa + \frac{2(\alpha\Gamma_1+\beta\Gamma_2)}{\sqrt{\Gamma_1^2 + \Gamma_2^2 + \theta^2}}\right)\left(\Gamma_3 + 2\theta - 2\sqrt{\Gamma_1^2 + \Gamma_2^2 + \theta^2} \right) \notag\\ 
    &\geq \frac{4\theta (\alpha\Gamma_1+\beta\Gamma_2)}{\sqrt{\Gamma_1^2 + \Gamma_2^2 + \theta^2}},
\end{align}
where the nonnegativity of $\eta$ is used.
Letting $\theta\to 0_+$, we get $\gamma(x,y,t)$ satisfying 
\begin{equation}\label{eq_gamma}
\begin{aligned}
\partial_{t}\gamma(x,y,t)+u\cdot\nabla\gamma(x,y,t)+2\left(\kappa+\frac{\alpha\Gamma_{1}+\Gamma_{2}\beta}{\sqrt{\Gamma^{2}_{1}+\Gamma^{2}_{2}}}\right)\gamma(x,y,t)
-\varepsilon\Delta\gamma(x,y,t)\geq0,
\end{aligned}
\end{equation}
with the initial and boundary conditions that
\begin{equation*} 
\begin{cases}
\gamma(x,y,0)=2\sqrt{\mathrm{\mathrm{det}}\mathbb{T}_0+\Gamma^{2}_{1}(x,y,0)+\Gamma^{2}_{2}(x,y,0)}-2\sqrt{\Gamma^{2}_{1}(x,y,0)
+\Gamma^{2}_{2}(x,y,0)}\geq0,\\
 \partial_{x}\gamma(0,y,t)=\partial_{x}\gamma(1,y,t)=0.
\end{cases} 
\end{equation*}
With the help of the regularity of $u$, we have
\begin{equation}\label{MT}
\begin{aligned}
\int^T_0{\Big\|\kappa+\frac{\alpha\Gamma_{1}+\Gamma_{2}\beta}{\sqrt{\Gamma^{2}_{1}+\Gamma^{2}_{2}}}\Big\|_{L^{\infty}_{xy}}}\mathrm{d}t
\leq C(\kappa T+ \|\nabla u\|_{L^1_{T}L^\infty_{xy}})\leq M_T.
\end{aligned}
\end{equation}
Setting $\mathfrak{a}=-\gamma$, $\mathfrak{a}_+ = \max\{0, \mathfrak{a}\}$, from \eqref{eq_gamma} we can get that
\begin{equation*}
\begin{aligned}
\frac{1}{2}\frac{\mathrm{d}}{\mathrm{d}t}\|\mathfrak{a}_+(x,y,t)\|^{2}_{L^{2}}+\varepsilon\|\nabla\mathfrak{a}_+(x,y,t)\|^{2}_{L^{2}}
\leq\Big\|\kappa+\frac{\alpha\Gamma_{1}+\Gamma_{2}\beta}{\sqrt{\Gamma^{2}_{1}+\Gamma^{2}_{2}}}\Big\|_{L^{\infty}_{xy}}
\|\mathfrak{a}_+(x,y,t)\|^{2}_{L^{2}},
\end{aligned}
\end{equation*}
which, together with Gronwall inequality, \eqref{MT} and the fact $\gamma_0\geq0$, implies
\begin{equation*}
\begin{aligned}
\|\mathfrak{a}_+(x,y,t)\|^{2}_{L^{2}}\leq\|\mathfrak{a}_+(x,y,0)\|^{2}_{L^{2}}\exp(M_T),
\end{aligned}
\end{equation*}
which implies $\mathfrak{a}_+\overset{a.e.}{=}0$. Therefore, 
$$\mathrm{tr}(\mathbb{T})=\Gamma_{3}\geq0,~~~ \mathrm{det}\mathbb{T}=\frac{\Gamma_{3}^2}{4}-(\Gamma^{2}_{1}+\Gamma^{2}_{2})\geq0.$$
The proof is completed.
\end{proof}
With the help of Lemma \ref{3.1}, we can deduce the following Corollary.
\begin{cor}\label{cor1}
Under the assumptions of Theorem \ref{thm4}, let $(u,\eta,\tau)$ be a strong solution  of problem \eqref{EQ-t1h}--\eqref{OB_as_boundary_1} with $\partial_x u\in L^2(0,T;H^2(\Omega))$ be divergence-free. Then it holds that   $\eta$ remains non-negative and $\mathbb{T}$ remains non-negative  definite. Moreover,
\begin{equation*} 
\|\mathbb{T}\|_{L^{1}}\leq\|\mathrm{tr}\mathbb{T}\|_{L^{1}} ,~~\|\tau\|_{L^1}\leq \|\mathrm{tr}\mathbb{T}\|_{L^{1}} + \sqrt{2}\|\eta\|_{L^1}. 
\end{equation*}
\end{cor}
\begin{proof}
From Lemma \ref{2.3} and $\partial_x u\in L^2(0,T;H^2(\Omega))$, we have
$$\|\nabla u\|_{L^1_TL^\infty} \leq \sqrt{T}(\|  u\|_{L^2H^2} + \|\partial_x u\|_{L^2H^2}) <\infty.$$
Then, the non-negativity of $\eta$ and $\mathbb{T}$ is a consequence of Lemma \ref{3.1}.
Moreover, it is clear, from Lemma \ref{3.1}, that
\begin{equation*}
\begin{aligned}
\mathrm{det}\mathbb{T}=\mathbb{T}_{11}\mathbb{T}_{22}-\mathbb{T}_{12}^{2}\geq0,
\end{aligned}
\end{equation*}
which means 
\begin{equation*}
\begin{aligned}
\mathbb{T}_{12}^{2}\leq\mathbb{T}_{11}\mathbb{T}_{22}.
\end{aligned}
\end{equation*}
Using the boundary conditions \eqref{OB_as_boundary_1}, we have
\begin{equation*}
\begin{aligned}
\|\mathbb{T}\|_{L^{1}}=&\int{(\mathbb{T}_{11}^{2}+2\mathbb{T}_{12}^{2} + \mathbb{T}_{22}^{2})^\frac{1}{2}}\mathrm{d}x\mathrm{d}y\\
\leq&\int{(\mathbb{T}_{11}^{2}+2\mathbb{T}_{11}\mathbb{T}_{22}+\mathbb{T}_{22}^{2})^\frac{1}{2}}\mathrm{d}x\mathrm{d}y\\
=&\int{(\mathbb{T}_{11}+\mathbb{T}_{22})}\mathrm{d}x\mathrm{d}y
=\|\mathrm{tr}\mathbb{T}\|_{L^{1}}.
\end{aligned}
\end{equation*}
Moreover, noticing the relation that $\tau = \mathbb{T} -\eta\mathbb{I}$, we can complete the proof by Minkowski inequality. 
\end{proof}
\section{Global well-posedness}\label{part2}
This section is devoted to the proof of Theorem \ref{thm4}.  For clarity of exposition, we confine our presentation to the spatial domain $\Omega=[0,1]\times\mathbb{R}$, while the periodic counterpart $\Omega=[0,1]\times\mathcal{T}$  can be addressable through slight modification  of our arguments, as indicated in Remark \ref{remark1}. We begin with the {\it a priori} estimates of the solutions.
\subsection{The estimates of \texorpdfstring{$(u,\tau,\eta)$}{(u,tau,eta)}}\label{20estun}
\begin{lem}\label{lem5.1}
Under the assumptions of Theorem \ref{thm4},  let $(u,\tau,\eta)$ be the solution of problem \eqref{EQ-t1h}--\eqref{OB_as_boundary_1} on $[0,T]$. Then, for any $t\in[0,T]$, we have
\begin{align*}
&\|\tau\|_{L^{1}}+\|u\|^{2}_{L^{2}}+\|\eta\|^{2}_{L^{2}}
+2\kappa\int_0^{t}{\|\tau\|_{L^{1}}}\mathrm{d}s
+2\mu\int_0^{t}{\|\partial_{x}u\|^{2}_{L^{2}}}\mathrm{d}s
+2\varepsilon\int_0^{t}{\|\nabla\eta\|^{2}_{L^{2}}}\mathrm{d}s
\leq J_{0},\\
&\|\tau\|^{2}_{L^{2}}+4\kappa\int_0^{t}{\|\tau\|^{2}_{L^{2}}}\mathrm{d}s+\varepsilon\int_0^{t}{\|\nabla\tau\|^{2}_{L^{2}}}\mathrm{d}s
\leq J_{1},
\end{align*}
where $J_{0}$ and $J_{1}$ are defined in \eqref{2u0} and \eqref{2tau0}  respectively.
\end{lem}
\begin{proof}
Multiplying $\eqref{EQ-t2h}_{1}$ and $\eqref{EQ-t2h}_{2}$ by $2u$ and $2\eta$ respectively, and taking the trace of $\eqref{EQ-t2h}_{3}$, then integrating the resulting equation over $[0,1]\times\mathbb{R}$, we have 

\begin{align}\label{eq_T_L1}
&\frac{\mathrm{d}}{\mathrm{d}t}\Big(\|\mathrm{tr}\mathbb{T}\|_{L^{1}}+\|u\|^{2}_{L^{2}}+\|\eta\|^{2}_{L^{2}}\Big)+2\kappa\|\mathrm{tr}\mathbb{T}\|_{L^{1}}
+2\mu\|\partial_x u\|^{2}_{L^{2}}+2\varepsilon\|\nabla\eta\|^{2}_{L^{2}} \notag\\
&\quad=4\kappa\int{\eta}\mathrm{d}x\mathrm{d}y = 4\kappa\|\eta_0\|_{L^1},
\end{align}
where the following fact is used:
\begin{equation*}
\begin{aligned}
\int{\mathrm{tr}(\nabla u\mathbb{T}+\mathbb{T}\nabla^\top u)}\mathrm{d}x\mathrm{d}y
=&2\int{(\partial_{x}u_{1}\mathbb{T}_{11}+\partial_{y}u_{2}\mathbb{T}_{22}+\partial_{x}u_{2}\mathbb{T}_{12}
+\partial_{y}u_{1}\mathbb{T}_{21})}\mathrm{d}x\mathrm{d}y\\
=&-2\int{\diver\mathbb{T}\cdot u}\mathrm{d}x\mathrm{d}y.
\end{aligned}
\end{equation*}
Integrating \eqref{eq_T_L1} over $[0,t]$, we get  
\begin{align}\label{2u0}
&\|\tau\|_{L^{1}}+\|u\|^{2}_{L^{2}}+\|\eta\|^{2}_{L^{2}}
+2\kappa\int_0^{t}{\|\tau\|_{L^{1}}}\mathrm{d}s
+2\mu\int_0^{t}{\|\partial_{x}u\|^{2}_{L^{2}}}\mathrm{d}s
+2\varepsilon\int_0^{t}{\|\nabla\eta\|^{2}_{L^{2}}}\mathrm{d}s\notag\\
\leq\,& \sqrt{2}\|\tau_0\|_{L^{1}}+(CT+1)\|\eta_0\|_{L^{1}}+\|\eta_0\|_{L^{2}}^2+\|u_0\|^{2}_{L^{2}}=:J_{0},
\end{align} 
where  the fact that $\tau=\mathbb{T}-\eta\mathbb{I}$ and Corollary \ref{cor1} are used.  

Next, multiplying $\eqref{EQ-t1h}_3$ by $\tau$ and integrating the result over ${[0,1]\times\mathbb{R}}$, we have 
\begin{align}\label{2tua}
&\frac{1}{2}\frac{\mathrm{d}}{\mathrm{d}t}\|\tau\|^{2}_{L^{2}}+2\kappa\|\tau\|^{2}_{L^{2}}+\varepsilon\|\nabla\tau\|^{2}_{L^{2}} \notag\\
&\quad=\int{((\nabla u)\tau+\tau\nabla^\top u):\tau}\mathrm{d}x\mathrm{d}y+ \int{\eta(\nabla u+\nabla^\top u):\tau}\mathrm{d}x\mathrm{d}y \notag\\
&\quad=:E_{1}+E_{2}.
\end{align} 
Thanks to Poincar\'{e} inequality and Lemma \ref{2.3}, and integration by parts, $E_1$ can be estimated as follows:
\begin{equation*}
\begin{aligned}
E_{1}=&-\int{(u_{l}\partial_{m}\tau_{ln}\tau_{nm}+u_{l}\tau_{ln}\partial_{m}\tau_{nm}+\partial_{n}\tau_{ml}u_{l}\tau_{nm}
+\tau_{ml}u_{l}\partial_{n}\tau_{nm})}\mathrm{d}x\mathrm{d}y\\
\leq&C\|\nabla\tau\|_{L^{2}}\|\tau\|^{\frac{1}{2}}_{L^{2}}\|\partial_{y}\tau\|^{\frac{1}{2}}_{L^{2}}\|u\|^{\frac{1}{2}}_{L^{2}}
(\|u\|_{L^{2}}+\|\partial_x u\|_{L^{2}})^{\frac{1}{2}}\\
\leq&\frac{\varepsilon}{4}\|\nabla\tau\|^{2}_{L^{2}}+C\|\tau\|^{2}_{L^{2}}\|u\|^{2}_{L^{2}}\|\partial_x u\|^{2}_{L^{2}}.
\end{aligned}
\end{equation*}
Similarly, for $E_{2}$, we have 
\begin{align*} 
E_{2}=&-\int{(\eta\partial_{n}\tau_{nm}u_{m}+\eta u_{n}\partial_{m}\tau_{nm})}\mathrm{d}x\mathrm{d}y
-\int{(\partial_{n}\eta\tau_{nm}u_{m}+\partial_{m}\eta u_{n}\tau_{nm})}\mathrm{d}x\mathrm{d}y\\
\leq&C\|\nabla\tau\|_{L^{2}}\|\eta\|^{\frac{1}{2}}_{L^{2}}\|\partial_{y}\eta\|^{\frac{1}{2}}_{L^{2}}\|u\|^{\frac{1}{2}}_{L^{2}}(\|u\|_{L^{2}}+\|\partial_x u\|_{L^{2}})^{\frac{1}{2}}\\
&+C\|\nabla\eta\|_{L^{2}}\|\tau\|^{\frac{1}{2}}_{L^{2}}\|\partial_y \tau\|^{\frac{1}{2}}_{L^{2}}\|u\|^{\frac{1}{2}}_{L^{2}}(\|u\|_{L^{2}}+\|\partial_x u\|_{L^{2}})^{\frac{1}{2}}\\
\leq&\frac{\varepsilon}{4}\|\nabla\tau\|^{2}_{L^{2}}+ \varepsilon\|\nabla\eta\|^{2}_{L^{2}}+C\big(\|\tau\|^{2}_{L^{2}}\|u\|^{2}_{L^{2}}\|\partial_x u\|^{2}_{L^{2}}
+\|\eta\|^{2}_{L^{2}}\|u\|^{2}_{L^{2}}\|\partial_x u\|^{2}_{L^{2}}\big).
\end{align*} 
Substituting above two inequalities into  \eqref{2tua}, we can obtain
\begin{equation*}
\begin{aligned}
\frac{\mathrm{d}}{\mathrm{d}t}\|\tau\|^{2}_{L^{2}}+4\kappa\|\tau\|^{2}_{L^{2}}+\varepsilon\|\nabla\tau\|^{2}_{L^{2}}
\leq 2\varepsilon\|\nabla\eta\|^{2}_{L^{2}}+C\big(\|\tau\|^{2}_{L^{2}}\|u\|^{2}_{L^{2}}\|\partial_x u\|^{2}_{L^{2}}
+\|\eta\|^{2}_{L^{2}}\|u\|^{2}_{L^{2}}\|\partial_x u\|^{2}_{L^{2}}\big),
\end{aligned}
\end{equation*}
which, with the help of Gronwall inequality and \eqref{2u0}, offers that 
\begin{align}\label{2tau0}
&\|\tau\|^{2}_{L^{2}}+4\kappa\int_0^{t}{\|\tau\|^{2}_{L^{2}}}\mathrm{d}s
+\varepsilon\int_0^{t}{\|\nabla\tau\|^{2}_{L^{2}}}\mathrm{d}s \notag\\
&\quad\leq C\mathrm{exp}\left(C\int_0^{t}{\|u\|^{2}_{L^{2}}\|\partial_x u\|^{2}_{L^{2}}}\mathrm{d}s\right)\left(\|\tau_{0}\|^{2}_{L^{2}}
+\int_0^{t}{(2\varepsilon\|\nabla\eta\|^{2}_{L^{2}}+
\|\eta\|^{2}_{L^{2}}\|u\|^{2}_{L^{2}}\|\partial_x u\|^{2}_{L^{2}})}\mathrm{d}s\right) \notag\\
&\quad\leq C\mathrm{e}^{CJ^{2}_{0}}\left(\|\tau_{0}\|^{2}_{L^{2}}+J_{0}+\frac{1}{\mu}J^{3}_{0}\right)=:J_1.
\end{align} 
The proof is finished by \eqref{2u0} and \eqref{2tau0}.
\end{proof}
\subsection{The first-order estimates of \texorpdfstring{$(u,\tau,\eta)$}{(u,tau,eta)}}\label{21estun}
\begin{lem}\label{lem5.2}
Under the assumptions of Theorem \ref{thm4},  let $(u,\tau,\eta)$ be the solution of problem \eqref{EQ-t1h}--\eqref{OB_as_boundary_1} on $[0,T]$. Then, for any $t\in[0,T]$, we have
\begin{align*}
&\|\nabla\eta\|^{2}_{L^{2}} +\|\nabla\tau\|^{2}_{L^{2}}+\mu\|\partial_x u\|^{2}_{L^{2}}+4\|\partial_y u\|^{2}_{L^{2}}+\int^t_0{\|\partial_t u\|^{2}_{L^{2}}}\mathrm{d}s+4\mu\int^t_0{\|\partial_{x}\partial_{y}u\|^{2}_{L^{2}}}\mathrm{d}s\\
&\quad\quad\quad
+4\kappa\int^t_0{\|\nabla\tau\|^{2}_{L^{2}}}\mathrm{d}s
+\varepsilon\int^t_0{\|\nabla^2\eta\|^{2}_{L^{2}}}\mathrm{d}s+\varepsilon\int^t_0{\|\nabla^2\tau\|^{2}_{L^{2}}}\mathrm{d}s \leq J_2,
\end{align*}
where $J_2$ is defined in \eqref{2tau1_1}.
\end{lem}
\begin{proof}
Multiplying $\eqref{EQ-t1h}_{1}$, $\eqref{EQ-t1h}_{1}$ and $\eqref{EQ-t1h}_{3}$ by $\partial_{t}u$, $-2\partial^2_{y}u$ and $-\Delta\tau$ respectively, summing the results up, then integrating the resulting equation over $[0,1]\times\mathbb{R}$, we obtain
\begin{align}\label{uxuy1}
&\frac{1}{2}\frac{\mathrm{d}}{\mathrm{d}t}(\|\nabla\tau\|^{2}_{L^{2}}+\mu\|\partial_x u\|^{2}_{L^{2}}+2\|\partial_y u\|^{2}_{L^{2}})
+\|\partial_t u\|^{2}_{L^{2}}+2\mu\|\partial_{x}\partial_{y} u\|^{2}_{L^{2}}
+2\kappa\|\nabla\tau\|^{2}_{L^{2}}+\varepsilon\|\nabla^2\tau\|^{2}_{L^{2}}\nonumber\\
&\quad=\int{\diver\tau\cdot \partial_t u}\mathrm{d}x\mathrm{d}y
-\int{u\cdot\nabla u\cdot \partial_t u}\mathrm{d}x\mathrm{d}y
-2\int{\diver\tau\cdot\partial^2_{y}u}\mathrm{d}x\mathrm{d}y\nonumber\\
&\quad\quad-2\int{u\cdot\nabla u\cdot\partial^2_{y}u}\mathrm{d}x\mathrm{d}y
+\int{u\cdot\nabla\tau:\Delta\tau}\mathrm{d}x\mathrm{d}y\nonumber\\
&\quad\quad-\int{((\nabla u)\tau+\tau\nabla^\top u):\Delta\tau}\mathrm{d}x\mathrm{d}y
-\int{\eta(\nabla u+\nabla^\top u):\Delta\tau}\mathrm{d}x\mathrm{d}y =:\sum_{i=1}^{7}F_i.
\end{align}
By H\"older inequality and Young inequality, we have
\begin{equation}\label{}
\begin{aligned}
F_1\leq C\|\nabla\tau\|_{L^{2}}\|\partial_t u\|_{L^{2}}\leq\frac{1}{10}\|\partial_t u\|^2_{L^{2}}+C\|\nabla\tau\|^2_{L^{2}}.
\end{aligned}
\end{equation}
The estimate of $F_2$ is more delicate. Notice that
\begin{equation}\label{F'2}
\begin{aligned}
F_2=&-\int{u_1\partial_{x}u_1\partial_{t}u_1}\mathrm{d}x\mathrm{d}y
-\int{u_1\partial_{x}u_2\partial_{t}u_2}\mathrm{d}x\mathrm{d}y\\
&-\int{u_2\partial_{y}u_1\partial_{t}u_1}\mathrm{d}x\mathrm{d}y
-\int{u_2\partial_{y}u_2\partial_{t}u_2}\mathrm{d}x\mathrm{d}y=:\sum_{i=1}^{4}F_{2i}.
\end{aligned}
\end{equation}
Then, by Lemma \ref{2.3},  Young  inequality, Poincar\'{e} inequality, and the fact $\partial_{y}u_2=-\partial_{x}u_1$, we can bound $F_{21}$--$F_{24}$ as follows:
\begin{align*}
F_{21}\leq&C\|\partial_{t}u_1\|_{L^{2}}\|\partial_{x}u_1\|^{\frac{1}{2}}_{L^{2}}\|\partial_{x}\partial_{y}u_1\|^{\frac{1}{2}}_{L^{2}}
\|u_1\|^{\frac{1}{2}}_{L^{2}}(\|u_1\|_{L^{2}}+\|\partial_{x}u_1\|_{L^{2}})^{\frac{1}{2}}\\
\leq&\frac{1}{10}\|\partial_{t}u\|^2_{L^{2}}+\frac{\mu}{8}\|\partial_{x}\partial_{y}u\|^{2}_{L^{2}}
+C\|\partial_{x}u\|^{4}_{L^{2}}\|u\|^2_{L^{2}},\\
F_{22}\leq&C\|\partial_{t}u_2\|_{L^{2}}\|\partial_{x}u_2\|^{\frac{1}{2}}_{L^{2}}\|\partial_{x}\partial_{y}u_2\|^{\frac{1}{2}}_{L^{2}}
\|u_1\|^{\frac{1}{2}}_{L^{2}}(\|u_1\|_{L^{2}}+\|\partial_{x}u_1\|_{L^{2}})^{\frac{1}{2}}\\
\leq&\frac{1}{10}\|\partial_{t}u\|^2_{L^{2}}+\frac{\mu}{8}\|\partial_{x}\partial_{y}u\|^{2}_{L^{2}}
+C\|\partial_{x}u\|^{4}_{L^{2}}\|u\|^2_{L^{2}},\\
F_{23}\leq&C\|\partial_{t}u_1\|_{L^{2}}\|u_2\|^{\frac{1}{2}}_{L^{2}}\|\partial_{y}u_2\|^{\frac{1}{2}}_{L^{2}}\|\partial_{y}u_1\|^{\frac{1}{2}}_{L^{2}}
(\|\partial_{y}u_1\|_{L^{2}}+\|\partial_{x}\partial_{y}u_1\|_{L^{2}})^{\frac{1}{2}}\\
\leq&\frac{1}{10}\|\partial_{t}u\|^2_{L^{2}}+\frac{\mu}{8}\|\partial_{x}\partial_{y}u\|^{2}_{L^{2}}
+C\|\partial_{y}u\|^{2}_{L^{2}}\|u\|^2_{L^{2}}\|\partial_{x}u\|^{2}_{L^{2}},\\
F_{24}\leq&C\|\partial_{t}u_2\|_{L^{2}}\|\partial^2_{y}u_2\|^{\frac{1}{2}}_{L^{2}}\|\partial_{y}u_2\|^{\frac{1}{2}}_{L^{2}}
\|u_2\|^{\frac{1}{2}}_{L^{2}}(\|u_2\|_{L^{2}}+\|\partial_{x}u_2\|_{L^{2}})^{\frac{1}{2}}\\
\leq&\frac{1}{10}\|\partial_{t}u\|^2_{L^{2}}+\frac{\mu}{8}\|\partial_{x}\partial_{y}u\|^{2}_{L^{2}}
+C\|\partial_{x}u\|^{4}_{L^{2}}\|u\|^2_{L^{2}}.
\end{align*}
Inserting above estimates of $F_{21}$--$F_{24}$ to \eqref{F'2}, we get  
\begin{equation}\label{eq_F2}
\begin{aligned}
F_{2}\leq \frac{2}{5}\|\partial_{t}u\|^2_{L^{2}}+\frac{\mu}{2}\|\partial_{x}\partial_{y}u\|^{2}_{L^{2}}
+C\|\partial_{x}u\|^{4}_{L^{2}}\|u\|^2_{L^{2}}+C\|\partial_{y}u\|^{2}_{L^{2}}\|u\|^2_{L^{2}}\|\partial_{x}u\|^{2}_{L^{2}}.
\end{aligned}
\end{equation}
$F_3$ can be bounded directly by 
\begin{equation}\label{F'3}
\begin{aligned}
F_3\leq C\|\nabla^2\tau\|_{L^{2}}\|\partial_y u\|_{L^{2}}
\leq\frac{\varepsilon}{8}\|\nabla^2\tau\|^2_{L^{2}}+C\|\partial_y u\|^2_{L^{2}}.
\end{aligned}
\end{equation}
Due to $\diver u=0$, from integration by parts, we have that   
\begin{align}\label{F'4}
F_4=&2\int{\partial_{y}u_i\partial_{i}u_j\partial_{y}u_j}\mathrm{d}x\mathrm{d}y
+\int u_i\partial_{i}\big(|\partial_y u_j|^2\big) \mathrm{d}x\mathrm{d}y\notag\\
=& 
2\int{\partial_{y}u_1\partial_{x}u_2\partial_{y}u_2}\mathrm{d}x\mathrm{d}y+2\int{\partial_{y}u_2\partial_{y}u_2\partial_{y}u_2}\mathrm{d}x\mathrm{d}y 
=: F_{41} + F_{42}.
\end{align}
By virtue of $\diver u=0$, $\|u\|_{L^{2}}\leq C\|\partial_{x}u\|_{L^{2}}$ and Lemma \ref{2.3} that
\begin{align}\label{F'41} 
F_{41}\leq&C\|\partial_{x}u_2\|_{L^{2}}\|\partial^2_{y}u_2\|^{\frac{1}{2}}_{L^{2}}\|\partial_{y}u_2\|^{\frac{1}{2}}_{L^{2}}
\|\partial_{y}u_1\|^{\frac{1}{2}}_{L^{2}}(\|\partial_{y}u_1\|_{L^{2}}+\|\partial_{x}\partial_{y}u_1\|_{L^{2}})^{\frac{1}{2}}\nonumber\\
\leq&C\|\partial_{x}u\|_{L^{2}}\|\partial_{x}\partial_{y}u\|^{\frac{1}{2}}_{L^{2}}
\|\partial_{y}u\|_{L^{2}}\|\partial_{x}\partial_{y}u\|^{\frac{1}{2}}_{L^{2}}\nonumber\\
\leq&\frac{\mu}{8}\|\partial_{x}\partial_{y}u\|^{2}_{L^{2}}
+C\|\partial_{x}u\|^{2}_{L^{2}}\|\partial_{y}u\|^{2}_{L^{2}},\nonumber\\
F_{42}\leq&C\|\partial_{y}u_2\|_{L^{2}}\|\partial^2_{y}u_2\|^{\frac{1}{2}}_{L^{2}}\|\partial_{y}u_2\|^{\frac{1}{2}}_{L^{2}}
\|\partial_{y}u_2\|^{\frac{1}{2}}_{L^{2}}(\|\partial_{y}u_2\|_{L^{2}}+\|\partial_{x}\partial_{y}u_2\|_{L^{2}})^{\frac{1}{2}}\nonumber\\
\leq&C\|\partial_{x}u\|_{L^{2}}\|\partial_{x}\partial_{y}u\|^{\frac{1}{2}}_{L^{2}}\|\partial_{y}u_2\|_{L^{2}}
\|\partial_{x}\partial_{y}u\|^{\frac{1}{2}}_{L^{2}}\nonumber\\
\leq&\frac{\mu}{8}\|\partial_{x}\partial_{y}u\|^{2}_{L^{2}}
+C\|\partial_{x}u\|^{2}_{L^{2}}\|\partial_{y}u\|^{2}_{L^{2}}.\notag
\end{align}
According to the estimate of $F_{41}$ and $F_{44}$, we obtain
\begin{equation}
\begin{aligned}
F_{4}\leq \frac{\mu}{4}\|\partial_{x}\partial_{y}u\|^{2}_{L^{2}}
+C\|\partial_{x}u\|^{2}_{L^{2}}\|\partial_{y}u\|^{2}_{L^{2}}.
\end{aligned}
\end{equation}
By Lemma \ref{2.3}, $F_{5}$ can be estimated directly:
\begin{align}\label{F'5}
F_{5}\leq&C\|\nabla^2\tau\|_{L^{2}}\|\nabla\tau\|^{\frac{1}{2}}_{L^{2}}\|\nabla\partial_{y}\tau\|^{\frac{1}{2}}_{L^{2}}
\|u\|^{\frac{1}{2}}_{L^{2}}(\|u\|_{L^{2}}+\|\partial_{x}u\|_{L^{2}})^{\frac{1}{2}}\notag\\
\leq&\frac{\varepsilon}{8}\|\nabla^2\tau\|^2_{L^{2}}+C\|\nabla\tau\|^{2}_{L^{2}}\|u\|^2_{L^{2}}\|\partial_{x}u\|^{2}_{L^{2}}.
\end{align} 
To bound $F_{6}$, we divide $F_{6}$ into two parts:
\begin{equation}\label{F'6}
\begin{aligned}
F_6=&-2\int \Big((\partial_{x}u_1\tau_{11}+\partial_{x}u_2\tau_{21})\Delta\tau_{11}
+\partial_{x}u_2\tau_{22}\Delta\tau_{21} \Big)\mathrm{d}x\mathrm{d}y\\
&-2\int{\Big((\partial_{y}u_1\tau_{21}+\partial_{y}u_2\tau_{22})\Delta\tau_{22}
+\partial_{y}u_1\tau_{11}\Delta\tau_{12}\Big)}\mathrm{d}x\mathrm{d}y=:F_{61}+F_{62}.
\end{aligned}
\end{equation}
Lemma \ref{2.3} and Young  inequality imply that
\begin{align*}
F_{61}\leq&C\|\nabla^2\tau\|_{L^{2}}\|\partial_{x}u\|^{\frac{1}{2}}_{L^{2}}\|\partial_{x}\partial_{y}u\|^{\frac{1}{2}}_{L^{2}}
\|\tau\|^{\frac{1}{2}}_{L^{2}}(\|\tau\|_{L^{2}}+\|\partial_{x}\tau\|_{L^{2}})^{\frac{1}{2}}\nonumber\\
\leq&\frac{\varepsilon}{16}\|\nabla^2\tau\|^2_{L^{2}}+\frac{\mu}{16}\|\partial_{x}\partial_{y}u\|^2_{L^{2}}
+C\|\partial_{x}u\|^2_{L^{2}}(\|\tau\|^4_{L^{2}}+\|\tau\|^2_{L^{2}}\|\nabla\tau\|^2_{L^{2}}),\nonumber\\
F_{62}\leq&C\|\nabla^2\tau\|_{L^{2}}\|\tau\|^{\frac{1}{2}}_{L^{2}}\|\partial_{y}\tau\|^{\frac{1}{2}}_{L^{2}}
\|\partial_{y}u\|^{\frac{1}{2}}_{L^{2}}(\|\partial_{y}u\|_{L^{2}}+\|\partial_{x}\partial_{y}u\|_{L^{2}})^{\frac{1}{2}}\nonumber\\
\leq&\frac{\varepsilon}{16}\|\nabla^2\tau\|^2_{L^{2}}+\frac{\mu}{16}\|\partial_{x}\partial_{y}u\|^2_{L^{2}}+C\|\partial_{y}u\|^2_{L^{2}}
\|\tau\|^2_{L^{2}}\|\nabla\tau\|^2_{L^{2}}.
\end{align*}
With the help of above two inequalities, $F_{6}$ can be bounded  by
\begin{equation*}
\begin{aligned}
F_{6}\leq\frac{\varepsilon}{8}\|\nabla^2\tau\|^2_{L^{2}}+\frac{\mu}{8}\|\partial_{x}\partial_{y}u\|^2_{L^{2}}
+C(\|\partial_{x}u\|^2_{L^{2}}+4\|\partial_{y}u\|^2_{L^{2}})(\|\tau\|^4_{L^{2}}+\|\tau\|^2_{L^{2}}\|\nabla\tau\|^2_{L^{2}}).
\end{aligned}
\end{equation*}
Similarly, we have
\begin{equation*}
\begin{aligned}
F_{7}\leq&\frac{\varepsilon}{8}\|\nabla^2\tau\|^2_{L^{2}}+\frac{\mu}{8}\|\partial_{x}\partial_{y}u\|^2_{L^{2}}+C(\|\partial_{x}u\|^2_{L^{2}}
+4\|\partial_{y}u\|^2_{L^{2}})(\|\eta\|^4_{L^{2}}+\|\eta\|^2_{L^{2}}\|\nabla\eta\|^2_{L^{2}}).
\end{aligned}
\end{equation*}
Inserting $F_{1}$--$F_{7}$ into \eqref{uxuy1}, we can obtain
\begin{align*}
&\frac{\mathrm{d}}{\mathrm{d}t}\Big(\|\nabla\tau\|^{2}_{L^{2}}+\mu\|\partial_x u\|^{2}_{L^{2}}+2\|\partial_y u\|^{2}_{L^{2}}\Big)+\|u_t\|^{2}_{L^{2}}+2\mu\|\partial_{x}\partial_{y} u\|^{2}_{L^{2}}
+4\kappa\|\nabla\tau\|^{2}_{L^{2}}+\varepsilon\|\nabla^2\tau\|^{2}_{L^{2}}\\
&\quad\leq C\Big(\|\nabla\tau\|^2_{L^{2}}+\mu\|\partial_x u\|^{2}_{L^{2}}+2\|\partial_y u\|^{2}_{L^{2}}\Big)\Big(C+\|u\|^2_{L^{2}}\|\partial_{x}u\|^{2}_{L^{2}}\\
&\quad\quad+\|\partial_{x}u\|^2_{L^{2}}+\|\tau\|^4_{L^{2}}
+\|\tau\|^2_{L^{2}}\|\nabla\tau\|^2_{L^{2}}
+\|\eta\|^2_{L^{2}}\|\nabla\eta\|^2_{L^{2}}+\|\eta\|^4_{L^{2}}\Big),
\end{align*}
which, together with Gronwall inequality, implies
\begin{align}\label{2tau1}
&\|\nabla\tau\|^{2}_{L^{2}}+\mu\|\partial_x u\|^{2}_{L^{2}}+2\|\partial_y u\|^{2}_{L^{2}}+\int^t_0{\|\partial_t u\|^{2}_{L^{2}}}\mathrm{d}s \notag\\
&\quad\quad\quad+2\mu\int^t_0{\|\partial_{x}\partial_{y}u\|^{2}_{L^{2}}}\mathrm{d}s
+4\kappa\int^t_0{\|\nabla\tau\|^{2}_{L^{2}}}\mathrm{d}s
+\varepsilon\int^t_0{\|\nabla^2\tau\|^{2}_{L^{2}}}\mathrm{d}s \notag\\
&\quad\leq C\mathrm{exp}\Big(\int^t_0{C+\|u\|^2_{L^{2}}\|\partial_{x}u\|^{2}_{L^{2}}+
\|\partial_{x}u\|^{2}_{L^{2}}+\|\tau\|^4_{L^{2}}
+\|\tau\|^2_{L^{2}}\|\nabla\tau\|^2_{L^{2}}} \notag\\
&\quad\quad{+\|\eta\|^2_{L^{2}}\|\nabla\eta\|^2_{L^{2}}
+\|\eta\|^4_{L^{2}}}\mathrm{d}s\Big)\Big(\|\nabla\tau_0\|^{2}_{L^{2}}+\|\nabla u_0\|^{2}_{L^{2}}\Big) \notag\\
&\quad\leq C\mathrm{exp}\Big(C(T+J_0^2+J_0+J_1^2+J_0^2T)\Big).
\end{align} 
Multiplying $\nabla\eqref{EQ-t1h}_{2}$ by $\nabla\eta$, integrating the result over $[0,1]\times\mathbb{R}$ by parts, we get  
\begin{align*}
\frac{1}{2}\frac{\mathrm{d}}{\mathrm{d}t}\|\nabla\eta\|^{2}_{L^{2}}+\varepsilon\|\nabla^2\eta\|^{2}_{L^{2}}
=\,&-\int{\nabla\eta\cdot\nabla u\cdot\nabla\eta}\mathrm{d}x\mathrm{d}y 
\leq\, C\|\nabla u\|_{L^{2}}\|\nabla\eta\|^2_{L^{4}}\\
\leq\,&C\|\nabla u\|_{L^{2}}\|\nabla\eta\|_{L^{2}}(\|\nabla\eta\|_{L^{2}}+\|\nabla^2\eta\|_{L^{2}})\\
\leq\,&\frac{\varepsilon}{2}\|\nabla^2\eta\|^2_{L^{2}}+C\|\nabla u\|^2_{L^{2}}\|\nabla\eta\|^2_{L^{2}}
+C\|\nabla u\|_{L^{2}}\|\nabla\eta\|^2_{L^{2}},
\end{align*} 
which implies 
\begin{equation}\label{eta2}
\begin{aligned}
\frac{\mathrm{d}}{\mathrm{d}t}\|\nabla\eta\|^{2}_{L^{2}}+\varepsilon\|\nabla^2\eta\|^{2}_{L^{2}}
\leq C\|\nabla u\|^2_{L^{2}}\|\nabla\eta\|^2_{L^{2}}+C\|\nabla u\|_{L^{2}}\|\nabla\eta\|^2_{L^{2}},
\end{aligned}
\end{equation}
Integrating \eqref{eta2} over $[0,t]$, it holds that
\begin{equation}\label{eq_eta_2}
\begin{aligned}
\|\nabla\eta\|^{2}_{L^{2}}+\varepsilon\int^t_0{\|\nabla^2\eta\|^{2}_{L^{2}}}\mathrm{d}s
\leq C\mathrm{exp}\Big(C(T+J_0^2+J_0+J_1^2+J_0^2T)\Big)J_0.
\end{aligned}
\end{equation}
Summing \eqref{2tau1} and \eqref{eq_eta_2}, we get  
\begin{align}\label{2tau1_1}
&\|\nabla\eta\|^{2}_{L^{2}} +\|\nabla\tau\|^{2}_{L^{2}}+\mu\|\partial_x u\|^{2}_{L^{2}}+2\|\partial_y u\|^{2}_{L^{2}}+\int^t_0{\|\partial_t u\|^{2}_{L^{2}}}\mathrm{d}s+2\mu\int^t_0{\|\partial_{x}\partial_{y}u\|^{2}_{L^{2}}}\mathrm{d}s \notag\\
&\quad\quad\quad
+4\kappa\int^t_0{\|\nabla\tau\|^{2}_{L^{2}}}\mathrm{d}s
+\varepsilon\int^t_0{\|\nabla^2\eta\|^{2}_{L^{2}}}\mathrm{d}s+\varepsilon\int^t_0{\|\nabla^2\tau\|^{2}_{L^{2}}}\mathrm{d}s \notag\\
&\quad\leq C\mathrm{exp}\Big(C(T+J_0^2+J_0+J_1^2+J_0^2T)\Big)\Big(1 +J_0\Big)=:J_2.
\end{align} 
The proof is complete.
\end{proof}
\subsection{The second-order estimates of \texorpdfstring{$(u,\tau,\eta)$}{u, tau,  eta}}\label{22estun}
In order to make the second-order estimate of $\tau$, $u$ and $\eta$, we first have the following result.
\begin{lem}\label{lem5.3}
Under the assumptions of Theorem \ref{thm4},  let $(u,\tau,\eta)$ be the solution of problem \eqref{EQ-t1h}--\eqref{OB_as_boundary_1} on $[0,T]$. Then, for any $t\in[0,T]$, we have
\begin{align*}
\|\partial_t\tau\|^{2}_{L^{2}}+\|\partial_t\eta\|^{2}_{L^{2}}+\|\partial_t u\|^{2}_{L^{2}}+&\mu\int^t_0{\|\partial_{t}\partial_{x}u\|^{2}_{L^{2}}}\mathrm{d}s
+\varepsilon\int^t_0{\|\nabla\partial_{t}\eta\|^{2}_{L^{2}}}\mathrm{d}s\\
+&4\kappa\int^t_0{\|\partial_t\tau\|^{2}_{L^{2}}}\mathrm{d}s
+\varepsilon\int^t_0{\|\nabla\partial_t\tau\|^{2}_{L^{2}}}\mathrm{d}s
\leq J_3,
\end{align*}
where $J_3$ is defined in \eqref{3t}.
\end{lem}
\begin{proof}
Multiplying $\partial_t\eqref{EQ-t1h}_1$, $\partial_t\eqref{EQ-t1h}_2$ and $\partial_t\eqref{EQ-t1h}_3$ by $\partial_t u$, $\partial_t \eta$ and $\partial_t \tau$ respectively, 
summing the results up, and
 integration the result over $[0,1]\times\mathbb{R}$ by parts, we have  
\begin{align}\label{utetat}
&\frac{1}{2}\frac{\mathrm{d}}{\mathrm{d}t}(\|\partial_t\tau\|^{2}_{L^{2}}
+\|\partial_t\eta\|^{2}_{L^{2}}+\|\partial_t u\|^{2}_{L^{2}})+\mu\|\partial_{t}\partial_{x}u\|^{2}_{L^{2}}
+\varepsilon\|\nabla\partial_{t}\eta\|^{2}_{L^{2}}
+2\kappa\|\partial_t\tau\|^{2}_{L^{2}}
+\varepsilon\|\nabla\partial_t\tau\|^{2}_{L^{2}}\nonumber\\
&\quad=\int{\diver\partial_{t}\tau\cdot\partial_{t}u}\mathrm{d}x\mathrm{d}y-
\int{\partial_{t}(u\cdot\nabla u)\cdot\partial_{t}u}\mathrm{d}x\mathrm{d}y -\int{\partial_{t}(u\cdot\nabla\eta)\cdot\partial_{t}\eta}
\mathrm{d}x\mathrm{d}y \nonumber\\
&\quad\quad-\int{\partial_{t}(u\cdot\nabla\tau):\partial_{t}\tau}\mathrm{d}x\mathrm{d}y+\int{\partial_t((\nabla u)\tau+\tau\nabla^\top u):\partial_t\tau}\mathrm{d}x\mathrm{d}y 
\nonumber\\
&\quad\quad+ \int{\partial_t(\eta(\nabla u+\nabla^\top u)):\partial_t\tau}\mathrm{d}x\mathrm{d}y=:\sum_{i=1}^{6}M_{i}.
\end{align}
For $M_1$, we have
\begin{equation} \label{eq_M_1}
M_1\leq C\|\nabla\partial_{t}\tau\|_{L^{2}}\|\partial_t u\|_{L^{2}}\leq\frac{\varepsilon}{8}\|\nabla\partial_{t}\tau\|^2_{L^{2}}+C\|\partial_t u\|^2_{L^{2}}.
\end{equation}
For $M_2$, from integration by parts, we have that
\begin{equation*}\label{}
\begin{aligned}
M_2=& -\int{\partial_{t}u_i\partial_{i}u_j\partial_{t}u_j}\mathrm{d}x\mathrm{d}y
-\frac{1}{2}\int{u_i\partial_{i} \big(| \partial_{t}u_j}|^2\big)\mathrm{d}x\mathrm{d}y\\
=&-\int{ \Big(\partial_{t}u_1\partial_{x}u_1\partial_{t}u_1
+\partial_{t}u_1\partial_{x}u_2\partial_{t}u_2-\partial_{t}u_2\partial_{x}u_1\partial_{t}u_2\Big)}\mathrm{d}x\mathrm{d}y\\
&-\int{\partial_{t}u_2\partial_{y}u_1\partial_{t}u_1}\mathrm{d}x\mathrm{d}y=:M_{21}+M_{22},
\end{aligned}
\end{equation*}
where the fact $\mathrm{div} u = 0$ is used.
With the help of Lemma \ref{2.3}, $\partial_{y}u_2=-\partial_{x}u_1$, Young inequality and Poincar\'{e} inequality, we have
\begin{align*}
M_{21}\leq&C\|\partial_{t}u\|_{L^{2}}\|\partial_{x}u\|^{\frac{1}{2}}_{L^{2}}\|\partial_{x}\partial_{y}u\|^{\frac{1}{2}}_{L^{2}}
\|\partial_{t}u\|^{\frac{1}{2}}_{L^{2}}(\|\partial_{t}u\|_{L^{2}}+\|\partial_{t}\partial_{x}u\|_{L^{2}})^{\frac{1}{2}}\nonumber\\
\leq&\frac{\mu}{16}\|\partial_{t}\partial_{x}u\|^2_{L^{2}}
+C\|\partial_{t}u\|^{2}_{L^{2}}\|\partial_{x}u\|_{L^{2}}\|\partial_{x}\partial_{y}u\|_{L^{2}},\nonumber\\
M_{22}\leq&C\|\partial_{t}u_1\|_{L^{2}}\|\partial_{t}u_2\|^{\frac{1}{2}}_{L^{2}}\|\partial_{t}\partial_yu_2\|^{\frac{1}{2}}_{L^{2}}
\|\partial_{y}u_1\|^{\frac{1}{2}}_{L^{2}}(\|\partial_{y}u_1\|_{L^{2}}+\|\partial_{x}\partial_{y}u_1\|_{L^{2}})^{\frac{1}{2}}\nonumber\\
\leq&\frac{\mu}{16}\|\partial_{t}\partial_{x}u\|^2_{L^{2}}
+C\|\partial_{t}u\|^{2}_{L^{2}}\|\partial_{y}u\|_{L^{2}}\|\partial_{x}\partial_{y}u\|_{L^{2}}.
\end{align*}
The estimations of $M_{21}$ and $M_{22}$ mean that
\begin{equation}  
M_{2}\leq\frac{\mu}{8}\|\partial_{t}\partial_{x}u\|^2_{L^{2}}+C\|\partial_{t}u\|^{2}_{L^{2}}(\|\partial_{y}u\|_{L^{2}}\|\partial_{x}\partial_{y}u\|_{L^{2}}
+\|\partial_{x}u\|_{L^{2}}\|\partial_{x}\partial_{y}u\|_{L^{2}}). 
\end{equation}
By virtue of Lemma \ref{2.3} and Young inequality, integration by parts, the estimate of $M_3$ and $M_4$ are directly given by 
\begin{align}
M_3 + M_4=&  \int{\partial_{t}u_i\eta\partial_{i}\partial_{t}\eta}\mathrm{d}x\mathrm{d}y + \int{\partial_{t}u_i\tau_{\ell m}\partial_{i}\partial_{t}\tau_{\ell m}}\mathrm{d}x\mathrm{d}y
 \notag\\
\leq&C\|\nabla\partial_{t}\eta\|_{L^{2}}\|\eta\|^{\frac{1}{2}}_{L^{2}}\|\partial_{y}\eta\|^{\frac{1}{2}}_{L^{2}}
\|\partial_{t}u\|^{\frac{1}{2}}_{L^{2}}(\|\partial_{t}u\|_{L^{2}}+\|\partial_{t}\partial_{x}u\|_{L^{2}})^{\frac{1}{2}}\notag\\
&+C\|\nabla\partial_{t}\tau\|_{L^{2}}\|\tau\|^{\frac{1}{2}}_{L^{2}}\|\partial_{y}\tau\|^{\frac{1}{2}}_{L^{2}}
\|\partial_{t}u\|^{\frac{1}{2}}_{L^{2}}(\|\partial_{t}u\|_{L^{2}}+\|\partial_{t}\partial_{x}u\|_{L^{2}})^{\frac{1}{2}} \notag\\ 
\leq&\frac{\varepsilon}{8}\Big(\|\nabla\partial_{t}\eta\|^2_{L^{2}} + \|\nabla\partial_{t}\tau\|^2_{L^{2}}\Big)+\frac{\mu}{8}\|\partial_{t}\partial_{x}u\|^2_{L^{2}} \notag\\
&+\|\partial_{t}u\|^{2}_{L^{2}}\|(\eta,\tau)\|^{2}_{L^{2}}\|\nabla(\eta,\tau)\|^{2}_{L^{2}}.
\end{align}
  Using integration by parts, we have
\begin{equation*} 
\begin{aligned}
M_5=& \int ((\partial_t\partial_{j}u_i\tau_{j\ell}  
+\tau_{ij}\partial_t\partial_{j}u_\ell )\partial_{t}\tau_{i\ell}  
+ (\partial_{j}u_i\partial_t\tau_{j\ell}  
+\partial_t\tau_{ij}\partial_{j}u_\ell )\partial_{t}\tau_{i\ell}) \mathrm{d}x\mathrm{d}y\\
=&  -\int (\partial_t u_i\tau_{j\ell}  
+\tau_{ij}\partial_t u_\ell )\partial_{j}\partial_{t}\tau_{i\ell}  \mathrm{d}x\mathrm{d}y -\int (\partial_t u_i\partial_{j}\tau_{j\ell}  
+\partial_{j}\tau_{ij}\partial_t u_\ell )\partial_{t} \tau_{i\ell}  \mathrm{d}x\mathrm{d}y \\
&+\int (\partial_{j}u_i\partial_t\tau_{j\ell}  
+\partial_t\tau_{ij}\partial_{j}u_\ell )\partial_{t}\tau_{i\ell} \mathrm{d}x\mathrm{d}y 
=:M_{51}+M_{52}+M_{53}.
\end{aligned}
\end{equation*}
Thanks to Lemma \ref{2.3} and Poincar\'{e} inequality, we can obtain
\begin{align*}
M_{51}\leq&C\|\nabla\partial_{t}\tau\|_{L^{2}}\|\tau\|^{\frac{1}{2}}_{L^{2}}\|\partial_{y}\tau\|^{\frac{1}{2}}_{L^{2}}
\|\partial_{t}u\|^{\frac{1}{2}}_{L^{2}}(\|\partial_{t}u\|_{L^{2}}+\|\partial_t\partial_{x}u\|_{L^{2}})^{\frac{1}{2}}\nonumber\\
\leq&\frac{\varepsilon}{24}\|\nabla\partial_{t}\tau\|^2_{L^{2}}+\frac{\mu}{16}\|\partial_t\partial_{x}u\|^2_{L^{2}}
+\|\partial_{t}u\|^{2}_{L^{2}}\|\tau\|^{2}_{L^{2}}\|\nabla\tau\|^{2}_{L^{2}},\nonumber\\
M_{52}\leq&C\|\nabla\tau\|_{L^{2}}\|\partial_{t}\tau\|^{\frac{1}{2}}_{L^{2}}\|\partial_t\partial_{y}\tau\|^{\frac{1}{2}}_{L^{2}}
\|\partial_{t}u\|^{\frac{1}{2}}_{L^{2}}(\|\partial_{t}u\|_{L^{2}}+\|\partial_t\partial_{x}u\|_{L^{2}})^{\frac{1}{2}}\nonumber\\
\leq&\frac{\varepsilon}{24}\|\nabla\partial_{t}\tau\|^2_{L^{2}}+\frac{\mu}{16}\|\partial_t\partial_{x}u\|^2_{L^{2}}
+\|\partial_{t}\tau\|^{2}_{L^{2}}\|\nabla\tau\|^4_{L^{2}},\nonumber\\
M_{53}\leq&C\|\nabla u\|_{L^{2}}\|\partial_{t}\tau\|^{\frac{1}{2}}_{L^{2}}\|\partial_t\partial_{y}\tau\|^{\frac{1}{2}}_{L^{2}}
\|\partial_{t}\tau \|^{\frac{1}{2}}_{L^{2}}(\|\partial_{t}\tau\|_{L^{2}}+\|\partial_{t}\partial_{x}\tau\|_{L^{2}})^{\frac{1}{2}}\nonumber\\
\leq&\frac{\varepsilon}{24}\|\nabla\partial_{t}\tau\|^2_{L^{2}} 
+C(1+\|\nabla u\|^2_{L^{2}})\|\partial_{t}\tau\|^{2}_{L^{2}}.
\end{align*}
Adding the estimates of $M_{51}$--$M_{53}$ up, we get
\begin{align}
M_{5}\leq&\frac{\varepsilon}{8}\|\nabla\partial_{t}\tau\|^2_{L^{2}}+\frac{\mu}{8}\|\partial_t\partial_{x}u\|^2_{L^{2}}
+C(\|\partial_{t}\tau\|^{2}_{L^{2}}+\|\partial_{t}u\|^{2}_{L^{2}})(\|u\|^{2}_{L^{2}}\|\partial_{x}u\|^{2}_{L^{2}} \notag\\
&+\|\tau\|^{2}_{L^{2}}\|\nabla\tau\|^{2}_{L^{2}}+\|\nabla\tau\|^4_{L^{2}} + \|\nabla u\|^{2}_{L^{2}} +1).
\end{align} 
Similarly to the analysis of $M_5$, we have 
\begin{equation*}\label{}
\begin{aligned}
M_6&=-2 \int{\partial_{t}u_j\eta\partial_t\partial_{i}\tau_{ij}}\mathrm{d}x\mathrm{d}y
-2 \int{\partial_{t}u_j\partial_{i}\eta\partial_{t}\tau_{ij}}\mathrm{d}x\mathrm{d}y
-2 \int{u_j\partial_{t}\eta\partial_t\partial_{i}\tau_{ij}}\mathrm{d}x\mathrm{d}y\\
&\quad-2 \int{u_j\partial_t\partial_{i}\eta\partial_{t}\tau_{ij}}\mathrm{d}x\mathrm{d}y\\
&=:M_{61}+M_{62}+M_{63}+M_{64},
\end{aligned}
\end{equation*}
where the symmetry of $\tau$ is used.
With the help of Lemma \ref{2.3}, the estimates 
 of $M_{61}$--$M_{64}$ are given by
\begin{align*}
M_{61}\leq&C\|\nabla\partial_{t}\tau\|_{L^{2}}\|\eta\|^{\frac{1}{2}}_{L^{2}}\|\partial_{y}\eta\|^{\frac{1}{2}}_{L^{2}}
\|\partial_{t}u\|^{\frac{1}{2}}_{L^{2}}(\|\partial_{t}u\|_{L^{2}}+\|\partial_{t}\partial_{x}u\|_{L^{2}})^{\frac{1}{2}}\nonumber\\
\leq&\frac{\varepsilon}{32}\|\nabla\partial_{t}\tau\|^2_{L^{2}}+\frac{\mu}{16}\|\partial_{t}\partial_{x}u\|^2_{L^{2}}
+C\|\partial_{t}u\|^{2}_{L^{2}}\|\eta\|^{2}_{L^{2}}\|\nabla\eta\|^{2}_{L^{2}},\nonumber\\
M_{62}\leq&C\|\nabla\eta\|_{L^{2}}\|\partial_{t}\tau\|^{\frac{1}{2}}_{L^{2}}\|\partial_{t}\partial_{y}\tau\|^{\frac{1}{2}}_{L^{2}}
\|\partial_{t}u\|^{\frac{1}{2}}_{L^{2}}(\|\partial_{t}u\|_{L^{2}}+\|\partial_{t}\partial_{x}u\|_{L^{2}})^{\frac{1}{2}}\nonumber\\
\leq&\frac{\varepsilon}{32}\|\nabla\partial_{t}\tau\|^2_{L^{2}}+\frac{\mu}{16}\|\partial_{t}\partial_{x}u\|^2_{L^{2}}
+C\|\partial_{t}\tau\|^{2}_{L^{2}}\|\nabla\eta\|^4_{L^{2}},\nonumber\\
M_{63}\leq&C\|\nabla\partial_{t}\tau\|_{L^{2}}\|\partial_{t}\eta\|^{\frac{1}{2}}_{L^{2}}\|\partial_{t}\partial_{y}\eta\|^{\frac{1}{2}}_{L^{2}}
\|u\|^{\frac{1}{2}}_{L^{2}}(\|u\|_{L^{2}}+\|\partial_{x}u\|_{L^{2}})^{\frac{1}{2}}\nonumber\\
\leq&\frac{\varepsilon}{32}\|\nabla\partial_{t}\tau\|^2_{L^{2}}+\frac{3\varepsilon}{16}\|\nabla\partial_{t}\eta\|^2_{L^{2}}
+C\|\partial_{t}\eta\|^{2}_{L^{2}}\|u\|^{2}_{L^{2}}\|\partial_{x}u\|^{2}_{L^{2}},\nonumber\\
M_{64}\leq&C\|\nabla\partial_{t}\eta\|_{L^{2}}\|\partial_{t}\tau\|^{\frac{1}{2}}_{L^{2}}\|\partial_{t}\partial_{y}\tau\|^{\frac{1}{2}}_{L^{2}}
\|u\|^{\frac{1}{2}}_{L^{2}}(\|u\|_{L^{2}}+\|\partial_{x}u\|_{L^{2}})^{\frac{1}{2}}\nonumber\\
\leq&\frac{\varepsilon}{32}\|\nabla\partial_{t}\tau\|^2_{L^{2}}+\frac{3\varepsilon}{16}\|\nabla\partial_{t}\eta\|^2_{L^{2}}
+C\|\partial_{t}\tau\|^{2}_{L^{2}}\|u\|^{2}_{L^{2}}\|\partial_{x}u\|^{2}_{L^{2}}.
\end{align*}
Summing the above four inequalities up, we get 
\begin{align}\label{eq_M_6}
M_{6}\leq&\frac{\varepsilon}{8}\|\nabla\partial_{t}\tau\|^2_{L^{2}}+\frac{3\varepsilon}{8}\|\nabla\partial_{t}\eta\|^2_{L^{2}}
+\frac{\mu}{8}\|\partial_{t}\partial_{x}u\|^2_{L^{2}}+(\|\partial_{t}u\|^{2}_{L^{2}}\nonumber\\
&+\|\partial_{t}\eta\|^{2}_{L^{2}}+\|\partial_{t}\tau\|^{2}_{L^{2}})(\|u\|^{2}_{L^{2}}\|\partial_{x}u\|^{2}_{L^{2}}
+\|\eta\|^{2}_{L^{2}}\|\nabla\eta\|^{2}_{L^{2}}+\|\nabla\eta\|^4_{L^{2}}).
\end{align}
Substituting \eqref{eq_M_1}--\eqref{eq_M_6} into \eqref{utetat}, we have 
\begin{align*}
&\frac{\mathrm{d}}{\mathrm{d}t}(\|\partial_t\tau\|^{2}_{L^{2}}+\|\partial_t\eta\|^{2}_{L^{2}}+\|\partial_t u\|^{2}_{L^{2}})
+\mu\|\partial_{t}\partial_{x}u\|^{2}_{L^{2}}+\varepsilon\|\nabla\partial_{t}\eta\|^{2}_{L^{2}}+4\kappa\|\partial_t\tau\|^{2}_{L^{2}}
+\varepsilon\|\nabla\partial_t\tau\|^{2}_{L^{2}}\\
\leq&C\Big(1+
\|\partial_{x}u\|_{L^{2}}^2+ \|\partial_{y}u\|_{L^{2}}^2 +\|\partial_{x}\partial_{y}u\|^2_{L^{2}}+\|\eta\|^{2}_{L^{2}}\|\nabla\eta\|^{2}_{L^{2}}+\|\tau\|^{2}_{L^{2}}\|\nabla\tau\|^{2}_{L^{2}}\\
&+\|\nabla\tau\|^4_{L^{2}}
+\|u\|^{2}_{L^{2}}\|\partial_{x}u\|^{2}_{L^{2}}+\|\nabla\eta\|^4_{L^{2}}\Big)\Big(\|\partial_t u\|^2_{L^{2}}+\|\partial_{t}\tau\|^{2}_{L^{2}}+\|\partial_{t}\eta\|^{2}_{L^{2}}\Big),
\end{align*} 
which, along with Gronwall inequality, implies  
\begin{align}\label{3t}
&\|\partial_t\tau\|^{2}_{L^{2}}+\|\partial_t\eta\|^{2}_{L^{2}}+\|\partial_t u\|^{2}_{L^{2}}+\int^t_0 \Big(\mu\|\partial_{t}\partial_{x}u\|^{2}_{L^{2}} +4\kappa \|\partial_t\tau\|^{2}_{L^{2}} 
+\varepsilon \|\nabla\partial_t(\eta,\tau)\|^{2}_{L^{2}}  \Big)\mathrm{d}s \notag\\
&\quad\leq C\mathrm{exp}\Big(\int^t_0{1+\|\partial_{x}u\|_{L^{2}}^2 + \|\partial_{y}u\|_{L^{2}}^2
+\|\partial_{x}\partial_{y}u\|^2_{L^{2}}+\|\eta\|^{2}_{L^{2}}\|\nabla\eta\|^{2}_{L^{2}}
+\|\tau\|^{2}_{L^{2}}\|\nabla\tau\|^{2}_{L^{2}}} \notag\\
&\quad\quad\quad{+\|\nabla\tau\|^4_{L^{2}}
+\|u\|^{2}_{L^{2}}\|\partial_{x}u\|^{2}_{L^{2}}
+\|\nabla\eta\|^4_{L^{2}}}\mathrm{d}s\Big)
\Big(\|\partial_t\tau_0\|^{2}_{L^{2}}
+\|\partial_t\eta_0\|^{2}_{L^{2}}+\|\partial_t u_0\|^{2}_{L^{2}}\Big) \notag\\
&\quad\leq C\mathrm{exp}(T+J_0 +J_1^2 + TJ_2+J_2+J_0^2+J^2_2 
) (\|\partial_t\tau_0\|^{2}_{L^{2}}
+\|\partial_t\eta_0\|^{2}_{L^{2}}
+\|\partial_t u_0\|^{2}_{L^{2}})=:J_3.
\end{align} 
The proof is complete.
\end{proof}
With the help of Lemma \ref{lem5.3}, we will make the second-order estimate of $\tau$, $u$ and $\eta$ as follows.
\begin{lem}\label{lem5.4}
Under the assumptions of Theorem \ref{thm4},  let $(u,\tau,\eta)$ be the solution of problem \eqref{EQ-t1h}--\eqref{OB_as_boundary_1} on $[0,T]$. Then, for any $t\in[0,T]$, we have
\begin{align*}
&\|\partial^2_y u\|^{2}_{L^{2}}+\|\partial_{x}\partial_{y}u\|^{2}_{L^{2}}+\|\nabla\partial_y\tau\|^{2}_{L^{2}}
+\mu\int^t_0{\|\partial^2_{y}\partial_{x}u\|^{2}_{L^{2}}}\mathrm{d}s
+\frac{1}{\mu}\int^t_0{\|\partial_{t}\partial_{y}u\|^{2}_{L^{2}}}\mathrm{d}s\\
&\quad\quad\quad\quad\quad\quad\quad\quad\quad\quad\quad\quad\quad\quad
+4\kappa\int^t_0{\|\nabla\partial_y\tau\|^{2}_{L^{2}}}\mathrm{d}s
+\varepsilon\int^t_0{\|\nabla^2\partial_y\tau\|^{2}_{L^{2}}}\mathrm{d}s
\leq J_4,\\
&\|\partial^2_x u\|_{L^{2}}^2 + 
\|\nabla^2\eta\|_{L^{2}}^2 +  \|\nabla^2\tau\|_{L^{2}}^2\leq J_5, \\
&\int^t_0 \Big(\|\nabla\partial_{x}^2u\|^{2}_{L^{2}} 
+ \|\nabla^3 \eta\|^{2}_{L^{2}} + \|\nabla^3 \tau\|^{2}_{L^{2}} \Big)\mathrm{d}s \leq J_6,
\end{align*}
where $J_4$, $J_5$ and $J_6$ are respectively defined in \eqref{2utau2}, \eqref{eq_222} and \eqref{eq_uetatau_3}.
\end{lem}
\begin{proof}
Multiplying $\partial^2_y\eqref{EQ-t1h}_1$, $\partial_y\eqref{EQ-t1h}_1$ and $\eqref{EQ-t1h}_3$ by $\partial^2_y u$, $\frac{1}{\mu}\partial_t\partial_{y} u$ and $\Delta\partial^2_{y} \tau$ respectively, integrating the resulting equation over $[0,1]\times\mathbb{R}$, we get
\begin{align}\label{2utau20}
&\frac{1}{2}\frac{\mathrm{d}}{\mathrm{d}t}(\|\partial^2_y u\|^{2}_{L^{2}}+\|\partial_{x}\partial_{y}u\|^{2}_{L^{2}}+\|\nabla\partial_y\tau\|^{2}_{L^{2}})
+\mu\|\partial^2_{y}\partial_{x}u\|^{2}_{L^{2}}
+\frac{1}{\mu}\|\partial_{t}\partial_{y}u\|^{2}_{L^{2}}\nonumber\\ 
&\quad+2\kappa\|\nabla\partial_y\tau\|^{2}_{L^{2}}
+\varepsilon\|\nabla^2\partial_y\tau\|^{2}_{L^{2}}\nonumber\\
\quad=&\int{\partial^2_{y}\diver\tau\cdot\partial^2_{y}u}\mathrm{d}x\mathrm{d}y-
\int{\partial^2_{y}(u\cdot\nabla u)\cdot\partial^2_{y}u}\mathrm{d}x\mathrm{d}y
+\frac{1}{\mu}\int{\diver\partial_{y}\tau\cdot\partial_{t}\partial_{y}u}
\mathrm{d}x\mathrm{d}y\nonumber\\
&  -\frac{1}{\mu}\int{\partial_{y}(u\cdot\nabla u)\cdot\partial_{t}\partial_{y}u}\mathrm{d}x\mathrm{d}y+
\int{((\nabla u)\tau+\tau\nabla^\top u):\Delta\partial^2_{y}\tau}\mathrm{d}x\mathrm{d}y\nonumber\\
&  + \int{(\eta(\nabla u+\nabla^\top u)):\Delta\partial^2_{y}\tau}\mathrm{d}x\mathrm{d}y
-\int{(u\cdot\nabla\tau):\Delta\partial^2_{y}\tau}\mathrm{d}x\mathrm{d}y =: \sum_{i=1}^{7}N_{i}.
\end{align}
By virtue of H\"{o}lder  inequality and Young  inequality, $N_1$ is directly estimated by
\begin{equation}\label{eq_N_1} 
N_1\leq C\|\nabla\partial^2_{y}\tau\|_{L^{2}}\|\partial^2_y u\|_{L^{2}}\leq\frac{\varepsilon}{8}\|\nabla^2\partial_{y}\tau\|^2_{L^{2}}+C\|\partial^2_y u\|^2_{L^{2}}. 
\end{equation}
Next, owing to the divergence free of $u$, we can split $N_2$ into two terms:
\begin{equation*}\label{}
\begin{aligned}
N_2=&-\int{\partial^2_{y}(u_i\partial_{i}u_j)\partial^2_{y}u_j}\mathrm{d}x\mathrm{d}y\\
=&-\int{\partial^2_{y}u_i\partial_{i}u_j\partial^2_{y}u_j}\mathrm{d}x\mathrm{d}y
-2\int{\partial_{y}u_i\partial_i\partial_{y}u_j\partial^2_{y}u_j}\mathrm{d}x\mathrm{d}y
=:N_{21}+N_{22}.
\end{aligned}
\end{equation*}
For $N_{21}$, we have
\begin{equation*}\label{}
\begin{aligned}
N_{21}=&-\int{\Big(\partial^2_{y}u_1\partial_{x}u_1\partial^2_{y}u_1
+\partial^2_{y}u_1\partial_{x}u_2\partial^2_{y}u_2
-\partial^2_{y}u_2\partial_{x}u_1\partial^2_{y}u_2\Big)}\mathrm{d}x\mathrm{d}y\\
&-\int{\partial^2_{y}u_2\partial_{y}u_1\partial^2_{y}u_1}\mathrm{d}x\mathrm{d}y
=:N_{211}+N_{212}.
\end{aligned}
\end{equation*}
Thanks to Lemma \ref{2.3},  Poincar\'{e} inequality and Young  inequality, $N_{211}$ and $N_{212}$ can be estimated by
\begin{align*}
N_{211}\leq&C\|\partial^2_{y}u\|_{L^{2}}\|\partial_{x}u\|^{\frac{1}{2}}_{L^{2}}\|\partial_{x}\partial_{y}u\|^{\frac{1}{2}}_{L^{2}}
\|\partial^2_{y}u\|^{\frac{1}{2}}_{L^{2}}(\|\partial^2_{y}u\|_{L^{2}}+\|\partial^2_{y}\partial_{x}u\|_{L^{2}})^{\frac{1}{2}}\\
\leq&\frac{\mu}{8}\|\partial^2_{y}\partial_{x}u\|^2_{L^{2}}+C\|\partial^2_{y}u\|^{2}_{L^{2}}
\|\partial_{x}u\|_{L^{2}}\|\partial_{x}\partial_{y}u\|_{L^{2}},\\
N_{212}\leq&C\|\partial^2_{y}u\|_{L^{2}}\|\partial^2_{y}u_2\|^{\frac{1}{2}}_{L^{2}}\|\partial^3_{y}u_2\|^{\frac{1}{2}}_{L^{2}}
\|\partial_{y}u_1\|^{\frac{1}{2}}_{L^{2}}(\|\partial_{y}u_1\|_{L^{2}}+\|\partial_{x}\partial_{y}u_1\|_{L^{2}})^{\frac{1}{2}}\\
\leq&\frac{\mu}{8}\|\partial^2_{y}\partial_{x}u\|^2_{L^{2}}+C\|\partial^2_{y}u\|^{2}_{L^{2}}\|\partial_{y}u_1\|^{2}_{L^{2}}\|\partial_{x}\partial_{y}u\|^{2}_{L^{2}}.
\end{align*}
Thus, for $N_{21}$, we have
\begin{equation*}\label{}
\begin{aligned}
N_{21}\leq&\frac{\mu}{4}\|\partial^2_{y}\partial_{x}u\|^2_{L^{2}}+C\|\partial^2_{y}u\|^{2}_{L^{2}}(\|\partial_{x}u\|_{L^{2}}\|\partial_{x}\partial_{y}u\|_{L^{2}}
+\|\partial_{y}u\|_{L^{2}}^2\|\partial_{x}\partial_{y}u\|_{L^{2}}^2).
\end{aligned}
\end{equation*}
Similarly, for $N_{22}$, we have  
\begin{align*}
N_{22}=&\,-2\int{\Big(\partial_{y}u_1\partial_{x}\partial_{y}u_1\partial^2_{y}u_1
+\partial_{y}u_1\partial_{x}\partial_{y}u_2\partial^2_{y}u_2
-\partial_{y}u_2\partial_{x}\partial_{y}u_1\partial^2_{y}u_2\Big)}\mathrm{d}x\mathrm{d}y\\
&-2\int{\partial_{y}u_2\partial_{y}^2u_1\partial^2_{y}u_1}\mathrm{d}x\mathrm{d}y\\ 
\leq&\,C\|\partial^2_{y}u\|_{L^{2}}\|\partial_{x}\partial_{y}u\|^{\frac{1}{2}}_{L^{2}}\|\partial_{y}^2\partial_{x}u\|^{\frac{1}{2}}_{L^{2}}
\|\partial_{y}u\|^{\frac{1}{2}}_{L^{2}}(\|\partial_{y}u\|_{L^{2}}+\|\partial_{x}\partial_{y}u\|_{L^{2}})^{\frac{1}{2}}\\
&+ \|\partial^2_{y}u\|_{L^{2}}\|\partial_{x}u_1\|^{\frac{1}{2}}_{L^{2}}\|\partial_{y}\partial_{x}u_1\|^{\frac{1}{2}}_{L^{2}}
\|\partial_{y}^2u\|^{\frac{1}{2}}_{L^{2}}(\|\partial_{y}^2u\|_{L^{2}}+\|\partial_{y}^2\partial_{x}u\|_{L^{2}})^{\frac{1}{2}}\\
\leq&\frac{\mu}{4}\|\partial^2_{y}\partial_{x}u\|^2_{L^{2}}+ C\|\partial_{x}\partial_{y}u\|^{4}_{L^{2}}
\|\partial_{y}u\|^{2}_{L^{2}} +C \|\partial^2_{y}u\|_{L^{2}}^2\|\partial_{x}u_1\|_{L^{2}}\|\partial_{y}\partial_{x}u_1\|_{L^{2}},
\end{align*} 
which, together with the estimate of  $N_{21}$, implies that
\begin{align}\label{eq_N_2} 
N_{2}\leq\,&\frac{\mu}{2}\|\partial^2_{y}\partial_{x}u\|^2_{L^{2}}
+C(\|\partial_{x}u\|_{L^{2}}\|\partial_{x}\partial_{y}u\|_{L^{2}}
+\|\partial_{y}u\|_{L^{2}}^2\|\partial_x\partial_{y}u\|_{L^{2}}^2
) \notag\\ 
&\times (\|\partial^2_{y}u\|^{2}_{L^{2}}+\|\partial_x\partial_{y}u\|_{L^{2}}^2).
\end{align}
By H\"{o}lder inequality and Young inequality, we can estimate $N_3$ as follows:
\begin{equation}\label{eq_N_3} 
N_3\leq C\|\nabla\partial_{y}\tau\|_{L^{2}}\|\partial_t\partial_{y} u\|_{L^{2}}\leq\frac{1}{4\mu}\|\partial_t\partial_{y}u\|^2_{L^{2}}+C\|\nabla^2\tau\|^2_{L^{2}}. 
\end{equation}
To bound $N_4$, we first split $N_4$ into two terms:
\begin{equation*}\label{}
\begin{aligned}
N_4=&-\frac{1}{\mu}\int{\partial_{y}u\cdot\nabla u\cdot\partial_{t}\partial_{y}u}\mathrm{d}x\mathrm{d}y-
\frac{1}{\mu}\int{u\cdot\nabla\partial_{y}u\cdot\partial_{t}\partial_{y}u}\mathrm{d}x\mathrm{d}y=:N_{41}+N_{42}.\\
\end{aligned}
\end{equation*} 
By Lemma \ref{2.3} and Young inequality, $N_{41}$ and $N_{42}$ can be bounded by
\begin{align*}
N_{41}\leq&C\|\partial_{t}\partial_{y}u\|_{L^{2}}\|\nabla u\|^{\frac{1}{2}}_{L^{2}}\|\nabla\partial_{y}u\|^{\frac{1}{2}}_{L^{2}}
\|\partial_{y}u\|^{\frac{1}{2}}_{L^{2}}(\|\partial_{y}u\|_{L^{2}}+\|\partial_{x}\partial_{y}u\|_{L^{2}})^{\frac{1}{2}}\\
\leq&\frac{1}{8\mu}\|\partial_{t}\partial_{y}u\|^2_{L^{2}}+C(\|\partial^2_{y}u\|^2_{L^{2}}+\|\partial_{x}\partial_{y}u\|^2_{L^{2}})\|\nabla u\|_{L^{2}}\|\partial_{x}\partial_{y}u\|_{L^{2}},\\
N_{42}\leq&C\|u\|_{L^{\infty}}\|\nabla\partial_{y}u\|_{L^{2}}\|\partial_{t}\partial_{y}u\|_{L^{2}}\\
\leq&\frac{1}{8\mu}\|\partial_{t}\partial_{y}u\|^2_{L^{2}}+C\|u\|^{\frac{1}{2}}_{L^{2}}\|\partial_{x}u\|^{\frac{1}{2}}_{L^{2}}\|\partial_{y}u\|^{\frac{1}{2}}_{L^{2}}
\|\partial_{x}\partial_{y}u\|^{\frac{1}{2}}_{L^{2}}(\|\partial^2_{y}u\|^2_{L^{2}}+\|\partial_{x}\partial_{y}u\|^2_{L^{2}})\\
\leq&\frac{1}{8\mu}\|\partial_{t}\partial_{y}u\|^2_{L^{2}}+C(\|\partial^2_{y}u\|^2_{L^{2}}+\|\partial_{x}\partial_{y}u\|^2_{L^{2}})\|\nabla u\|_{L^{2}}\|\partial_{x}\partial_{y}u\|_{L^{2}}.
\end{align*}
Summing above two estimates, we have
\begin{equation}\label{eq_N_4} 
N_{4}\leq \frac{1}{4\mu}\|\partial_{t}\partial_{y}u\|^2_{L^{2}}+C(\|\partial^2_{y}u\|^2_{L^{2}}+\|\partial_{x}\partial_{y}u\|^2_{L^{2}})\|\nabla u\|_{L^{2}}\|\partial_{x}\partial_{y}u\|_{L^{2}}. 
\end{equation}
In order to estimate $N_5$, we write $N_5$ more explicitly:
\begin{equation*}\label{}
\begin{aligned}
N_5=&\, \int{(\partial_{\ell}u_i\tau_{\ell j} + \tau_{i\ell }\partial_{\ell}u_j)\Delta\partial^2_{y}\tau_{ij}}\mathrm{d}x\mathrm{d}y\\
=&-\int(\partial_{\ell}u_i\partial_{y}\tau_{\ell j} + \partial_{y}\tau_{i\ell }\partial_{\ell}u_j)\Delta\partial_{y}\tau_{ij}\mathrm{d}x\mathrm{d}y\\
&-\int(\partial_{y}\partial_{\ell}u_i\tau_{\ell j} + \tau_{i\ell }\partial_{y}\partial_{\ell}u_j)\Delta\partial_{y}\tau_{ij}\mathrm{d}x\mathrm{d}y
=:\,N_{51}+N_{52}.
\end{aligned}
\end{equation*}
The estimates for $N_{51}$ and $N_{52}$ are as follows:
\begin{align*}
N_{51}\leq&C\|\nabla^2\partial_{y}\tau\|_{L^{2}}\|\nabla u\|^{\frac{1}{2}}_{L^{2}}\|\nabla\partial_{y}u\|^{\frac{1}{2}}_{L^{2}}
\|\partial_{y}\tau\|^{\frac{1}{2}}_{L^{2}}(\|\partial_{y}\tau\|_{L^{2}}+\|\partial_x\partial_{y}\tau\|_{L^{2}})^{\frac{1}{2}}\\
\leq&\frac{\varepsilon}{16}\|\nabla^2\partial_{y}\tau\|^2_{L^{2}}+\|\nabla u\|_{L^{2}}(\|\partial^2_{y}u\|_{L^{2}}+\|\partial_x\partial_{y}u\|_{L^{2}})
\|\partial_{y}\tau\|_{L^{2}}(\|\partial_{y}\tau\|_{L^{2}}+\|\partial_x\partial_{y}\tau\|_{L^{2}})\\
\leq&\frac{\varepsilon}{16}\|\nabla^2\partial_{y}\tau\|^2_{L^{2}}+C(\|\partial^2_{y}u\|^2_{L^{2}}+\|\partial_x\partial_{y}u\|^2_{L^{2}})
\|\nabla u\|^2_{L^{2}}(\|\nabla \tau\|^2_{L^{2}}+\|\nabla^2 \tau\|^2_{L^{2}}),\\
N_{52}\leq&C\|\tau\|_{L^{\infty}}\|\nabla\partial_{y}u\|_{L^{2}}\|\nabla^2\partial_{y}\tau\|_{L^{2}}\\
\leq&\frac{\varepsilon}{16}\|\nabla^2\partial_{y}\tau\|^2_{L^{2}}+C\|\tau\|^{\frac{1}{2}}_{L^{2}}(\|\tau\|_{L^{2}}+\|\partial_{x}\tau\|_{L^{2}})^{\frac{1}{2}}\|\partial_{y}\tau\|^{\frac{1}{2}}_{L^{2}}
(\|\partial_{y}\tau\|_{L^{2}}+\|\partial_{x}\partial_{y}\tau\|_{L^{2}})^{\frac{1}{2}}\|\nabla\partial_{y}u\|^2_{L^{2}}\\
\leq&\frac{\varepsilon}{16}\|\nabla^2\partial_{y}\tau\|^2_{L^{2}}+C(\|\partial^2_{y}u\|^2_{L^{2}}+\|\partial_{x}\partial_{y}u\|^2_{L^{2}})
(\|\tau\|_{L^{2}}\|\nabla \tau\|_{L^{2}}+\|\tau\|_{L^{2}}\|\nabla \tau\|^\frac{1}{2}_{L^{2}}\|\nabla^2\tau\|^\frac{1}{2}_{L^{2}}\\
&+\|\tau\|^\frac{1}{2}_{L^{2}}\|\nabla\tau\|^\frac{3}{2}_{L^{2}}
+\|\tau\|^\frac{1}{2}_{L^{2}}\|\nabla\tau\|_{L^{2}}\|\nabla^2\tau\|^\frac{1}{2}_{L^{2}}).
\end{align*}
Summing the above two inequalities up, we get  
\begin{align}
N_{5}\leq&\frac{\varepsilon}{8}\|\nabla^2\partial_{y}\tau\|^2_{L^{2}}+\|\nabla\tau\|^2_{L^{2}}+C(\|\partial^2_{y}u\|^2_{L^{2}}
+\|\partial_{x}\partial_{y}u\|^2_{L^{2}})(\|\nabla u\|^2_{L^{2}}\|\nabla \tau\|^2_{L^{2}} \notag \\
&+\|\nabla u\|^2_{L^{2}}\|\nabla^2\tau\|^2_{L^{2}}+\|\tau\|_{L^{2}}^2 + \|\nabla \tau\|_{L^{2}}^2 + \|\nabla^2  \tau\|_{L^{2}}^2).
\end{align} 
Similarly, for $N_{6}$, we have
\begin{align}
N_{6}\leq&\frac{\varepsilon}{8}\|\nabla^2\partial_{y}\tau\|^2_{L^{2}}+\|\nabla\eta\|^2_{L^{2}}+C(\|\partial^2_{y}u\|^2_{L^{2}}
+\|\partial_{x}\partial_{y}u\|^2_{L^{2}})(\|\nabla u\|^2_{L^{2}}\|\nabla \eta\|^2_{L^{2}} \notag\\
&+\|\nabla u\|^2_{L^{2}}\|\nabla^2 \eta\|^2_{L^{2}}+\|\eta\|_{L^{2}}^2 + \|\nabla \eta\|_{L^{2}}^2 + \|\nabla^2  \eta\|_{L^{2}}^2).
\end{align} 
Notice that $N_7$ can be written as
\begin{equation*}\label{}
\begin{aligned}
N_7=&-\int{u_i\partial_{i}\tau_{nj}\Delta\partial^2_{y}\tau_{nj}}\mathrm{d}x\mathrm{d}y\\
=&\int{u_n\partial_i\partial_{y}\tau_{nj}\Delta\partial_{y}\tau_{nj}}\mathrm{d}x\mathrm{d}y
+\int{\partial_{y}u_n\partial_{i}\tau_{nj}\Delta\partial_{y}\tau_{nj}}\mathrm{d}x\mathrm{d}y=:N_{71}+N_{72}.
\end{aligned}
\end{equation*}
With the help of Lemma \ref{2.3} and Young inequality, $N_{71}$ and $N_{72}$ are bounded by
\begin{align*}
N_{71}\leq&C\|\nabla^2\partial_{y}\tau\|_{L^{2}}\|\nabla\partial_{y}\tau\|^{\frac{1}{2}}_{L^{2}}\|\nabla\partial^2_{y}\tau\|^{\frac{1}{2}}_{L^{2}}
\|u\|^{\frac{1}{2}}_{L^{2}}(\|u\|_{L^{2}}+\|\partial_{x}u\|_{L^{2}})^{\frac{1}{2}}\\
\leq&\frac{\varepsilon}{16}\|\nabla^2\partial_{y}\tau\|^2_{L^{2}}+C\|\nabla\partial_{y}\tau\|^2_{L^{2}}
\|u\|^2_{L^{2}}\|\partial_{x}u\|^2_{L^{2}},\\
N_{72}\leq&C\|\nabla^2\partial_{y}\tau\|_{L^{2}}\|\partial_{y}u\|^{\frac{1}{2}}_{L^{2}}\|\partial^2_{y}u\|^{\frac{1}{2}}_{L^{2}}
\|\nabla\tau\|^{\frac{1}{2}}_{L^{2}}(\|\nabla\tau\|_{L^{2}}+\|\nabla\partial_{x}\tau\|_{L^{2}})^{\frac{1}{2}}\\
\leq&\frac{\varepsilon}{16}\|\nabla^2\partial_{y}\tau\|^2_{L^{2}} + C(\|\partial_{x}\partial_{y}u\|^2_{L^{2}}+\|\partial^2_{y} u\|^2_{L^{2}})(\|\nabla \tau\|_{L^{2}}\|\nabla^2 \tau\|_{L^{2}}+\|\nabla \tau\|^2_{L^{2}}).
\end{align*}
Substituting $N_{71}$ and $N_{72}$ into $N_{7}$  means that 
\begin{align}\label{eq_N_7}
N_{7}\leq&\frac{\varepsilon}{8}\|\nabla^2\partial_{y}\tau\|^2_{L^{2}}+C(\|\partial_{x}\partial_{y}u\|^2_{L^{2}}+\|\partial^2_{y} u\|^2_{L^{2}}+\|\nabla\partial_{y}\tau\|^2_{L^{2}})(\|u\|^2_{L^{2}}\|\partial_{x}u\|^2_{L^{2}} \notag\\
&+\|\nabla \tau\|_{L^{2}}\|\nabla^2 \tau\|_{L^{2}}+\|\nabla \tau\|^2_{L^{2}}).
\end{align}
Inserting \eqref{eq_N_1}--\eqref{eq_N_7} into \eqref{2utau20}, we have
\begin{align*}
&\frac{\mathrm{d}}{\mathrm{d}t}(\|\partial^2_y u\|^{2}_{L^{2}}+\|\partial_{x}\partial_{y}u\|^{2}_{L^{2}}+\|\nabla\partial_y\tau\|^{2}_{L^{2}})
+\mu\|\partial^2_{y}\partial_{x}u\|^{2}_{L^{2}}
+\frac{1}{\mu}\|\partial_{t}\partial_{y}u\|^{2}_{L^{2}}\\
&+4\kappa\|\nabla\partial_y\tau\|^{2}_{L^{2}}
+\varepsilon\|\nabla^2\partial_y\tau\|^{2}_{L^{2}}\\
\leq& C
\Big(1+\|\partial_{x}\partial_{y}u\|^{2}_{L^{2}}+\|\nabla u\|_{L^{2}}^2\|\partial_{x}\partial_{y}u\|_{L^{2}}^2+\|\nabla u\|^2_{L^{2}}\|\nabla \tau\|^2_{L^{2}}+\|\nabla u\|^2_{L^{2}}\|\nabla^2 \tau\|^2_{L^{2}}\\
&+\|\tau\|_{L^{2}}^2+\|\nabla \tau\|_{L^{2}}^2+\|\nabla^2 \tau\|_{L^{2}}^2
+\|\eta\|_{L^{2}}^2+\|\nabla \eta\|_{L^{2}}^2+\|\nabla^2 \eta\|_{L^{2}}^2+\|\nabla u\|^2_{L^{2}}\|\nabla \eta\|^2_{L^{2}}\\
&
 +\|\nabla u\|^2_{L^{2}}\|\nabla^2 \eta\|^2_{L^{2}}
+\|u\|^2_{L^{2}}\|\partial_{x}u\|^2_{L^{2}}\Big)\Big(\|\partial^2_y u\|^{2}_{L^{2}}+\|\partial_{x}\partial_{y}u\|^{2}_{L^{2}}+\|\nabla\partial_y\tau\|^{2}_{L^{2}}\Big)+C\|\nabla^2  \tau\|_{L^{2}}^2,
\end{align*}
which, together with Gronwall inequality, holds
\begin{align}\label{2utau2}
&\|\partial^2_y u\|^{2}_{L^{2}}+\|\partial_{x}\partial_{y}u\|^{2}_{L^{2}}+\|\nabla\partial_y\tau\|^{2}_{L^{2}}
+\mu\int^t_0{\|\partial^2_{y}\partial_{x}u\|^{2}_{L^{2}}}\mathrm{d}s+\frac{1}{\mu}\int^t_0{\|\partial_{t}\partial_{y}u\|^{2}_{L^{2}}}\mathrm{d}s\nonumber\\
& 
+4\kappa\int^t_0{\|\nabla\partial_y\tau\|^{2}_{L^{2}}}\mathrm{d}s
+\varepsilon\int^t_0{\|\nabla^2\partial_y\tau\|^{2}_{L^{2}}}\mathrm{d}s\nonumber\\
\leq &\,C\mathrm{exp}\Big(T+J_2+J_2^2  +J_2J_1          + J_1  +J_2 +J_0 + J_0T + J_0J_2 + J_0^2\Big)\nonumber\\
& \times(\|\partial^2_y u_0\|^{2}_{L^{2}}+\|\partial_{x}\partial_{y}u_0\|^{2}_{L^{2}}
+\|\nabla\partial_y\tau_0\|^{2}_{L^{2}}+J_2)=:J_{4}.
\end{align}
Next, we will estimate $\|\partial^2_x u\|_{L^{2}}$, $\|\nabla^2\eta\|_{L^{2}}$ and $\|\nabla^2\tau\|_{L^{2}}$ in sequence. 
Using $\eqref{EQ-t1h}_1$, we have  
\begin{align}\label{u_xx}
\|\partial^2_x u\|_{L^{2}}\leq & C(\|\nabla\tau\|_{L^{2}}+\|u\cdot\nabla u\|_{L^{2}}+\|\partial_t u\|_{L^{2}} +\|u\|_{L^{2}}) \notag \\ 
\leq&C\|\nabla u\|_{L^{2}}\|u\|^{\frac{1}{4}}_{L^{2}}(\|u\|_{L^{2}}+\|\partial_{x}u\|_{L^{2}})^{\frac{1}{4}}
\|\partial_{y}u\|^{\frac{1}{4}}_{L^{2}}(\|\partial_{y}u\|_{L^{2}} \notag \\
&+\|\partial_{x}\partial_{y}u\|_{L^{2}})^{\frac{1}{4}}+C\|\partial_t u\|_{L^{2}}+C\|\nabla\tau\|_{L^{2}}  +C\| u\|_{L^{2}} \notag \\
\leq&C[J^\frac{1}{2}_2J^\frac{1}{8}_0(J^\frac{1}{2}_{0}+J^\frac{1}{2}_2)^{\frac{1}{4}}  J_2^{\frac{1}{8}}(J^\frac{1}{2}_2+J^\frac{1}{2}_4)^{\frac{1}{4}}+J^\frac{1}{2}_0+J^\frac{1}{2}_2+J^\frac{1}{2}_3]=: J_{51}.
\end{align} 
With the help of $\eqref{EQ-t1h}_2$, Lemma \ref{2.3}, Lemma \ref{lem5.1}--\ref{lem5.3}, \eqref{2utau2}, 
 and the elliptic estimates, we have
\begin{align}\label{2eta2}
\|\nabla^2\eta\|_{L^{2}}\leq&C(\|u\cdot\nabla\eta\|_{L^{2}}+\|\partial_t\eta\|_{L^{2}} + \|\eta\|_{L^{2}})\notag\\
\leq&C\|\nabla\eta\|_{L^{2}}\|u\|^{\frac{1}{4}}_{L^{2}}(\|u\|_{L^{2}}+\|\partial_{x}u\|_{L^{2}})^{\frac{1}{4}}
\|\partial_{y}u\|^{\frac{1}{4}}_{L^{2}}(\|\partial_{y}u\|_{L^{2}} \notag \\
&+\|\partial_{x}\partial_{y}u\|_{L^{2}})^{\frac{1}{4}}+C\|\partial_t\eta\|_{L^{2}} + C\|\eta\|_{L^{2}}\notag\\
\leq& C[J^\frac{1}{2}_2J^\frac{1}{8}_0(J^\frac{1}{2}_{0}+J^\frac{1}{2}_2)^{\frac{1}{4}}  J_2^{\frac{1}{8}}(J^\frac{1}{2}_2+J^\frac{1}{2}_4)^{\frac{1}{4}}+J^\frac{1}{2}_0+J^\frac{1}{2}_3]=:J_{52}.
\end{align}
Similarly to the analysis of $\|\nabla^2\eta\|_{L^{2}}$, we have
\begin{align} \label{2tau2}
&\|\nabla^2\tau\|_{L^{2}}\notag\\ \leq&C\Big(\|u\cdot\nabla\tau\|_{L^{2}}+\|\partial_t\tau\|_{L^{2}}+\|\tau\|_{L^{2}}
+\|\tau\nabla u\|_{L^{2}}+\|\eta\nabla u\|_{L^{2}} \Big) \notag \\
\leq&C\Big(\|\partial_t\tau\|_{L^{2}}+\|\tau\|_{L^{2}}+C\|\nabla\tau\|_{L^{2}}\|u\|_{L^{\infty}}+C\|\nabla u\|_{L^{2}}(\|\eta\|_{L^{\infty}}+\|\tau\|_{L^{\infty}})\Big) \notag  \\
\leq&C\|\partial_t\tau\|_{L^{2}}+ C\|\tau\|_{L^{2}} \notag\\
&+ C\|\nabla\tau\|_{L^{2}}\|u\|^{\frac{1}{4}}_{L^{2}}(\|u\|_{L^{2}}+\|\partial_{x}u\|_{L^{2}})^{\frac{1}{4}}
\|\partial_{y}u\|^{\frac{1}{4}}_{L^{2}}(\|\partial_{y}u\|_{L^{2}} 
+\|\partial_{x}\partial_{y}u\|_{L^{2}})^{\frac{1}{4}}
 \notag  \\
&+C\|\nabla u\|_{L^{2}}\|\tau\|^{\frac{1}{4}}_{L^{2}}(\|\tau\|_{L^{2}}+\|\partial_{x}\tau\|_{L^{2}})^{\frac{1}{4}}
\|\partial_{y}\tau\|^{\frac{1}{4}}_{L^{2}}(\|\partial_{y}\tau\|_{L^{2}}+\|\partial_{x}\partial_{y}\tau\|_{L^{2}})^{\frac{1}{4}} \notag 
 \\
&+C\|\nabla u\|_{L^{2}}\|\eta\|^{\frac{1}{4}}_{L^{2}}(\|\eta\|_{L^{2}}+\|\partial_{x}\eta\|_{L^{2}})^{\frac{1}{4}}
\|\partial_{y}\eta\|^{\frac{1}{4}}_{L^{2}}(\|\partial_{y}\eta\|_{L^{2}}+\|\partial_{x}\partial_{y}\eta\|_{L^{2}})^{\frac{1}{4}} \notag \\
\leq&C\Big[J^\frac{1}{2}_3 + J^\frac{1}{2}_1 +J^\frac{5}{8}_2 J^\frac{1}{8}_{0}(J^\frac{1}{8}_{0}+ J_2^{\frac{1}{8}})(J_2^{\frac{1}{8}} + J_4^{\frac{1}{8}}+J^{\frac{1}{4}}_{52})+J^\frac{5}{8}_2 J^\frac{1}{8}_{1}(J^\frac{1}{8}_{1}+ J_2^{\frac{1}{8}})(J_2^{\frac{1}{8}} + J_4^{\frac{1}{8}})\Big] =:J_{53}.
\end{align}
Summing \eqref{u_xx}, \eqref{2eta2} and \eqref{2tau2} up, we find that   
\begin{equation}\label{eq_222}
    \|\partial^2_x u\|_{L^{2}}^2 + 
\|\nabla^2\eta\|_{L^{2}}^2 
 + \|\nabla^2\tau\|_{L^{2}}^2\leq J_{51}^2 + J_{52}^2 + J_{53}^2 =:J_5.
\end{equation} 
Next, using \eqref{EQ-t1h}$_1$, Lemma \ref{lem5.1}--\ref{lem5.3}, \eqref{2utau2} and \eqref{eq_222}, we have 
\begin{align}\label{eq_nablaxxu}
\int_0^t\|\nabla\partial_x^2 u\|_{L^2}^2 \mathrm{d}s
\leq&\, C\int_0^t (\|\nabla(u\cdot\nabla u)\|_{L^2}^2 + \|\nabla\partial_t u \|_{L^2}^2 + \|\nabla \mathrm{div} u\|_{L^2}^2 + \|u\|_{H^1}^2)\mathrm{d}s \notag \\ 
\leq&\, C\int_0^t\Big( (\|\nabla u\|_{L^2}^2+\|\partial_xu\|_{L^2}\|\partial_x\partial_yu\|_{L^2})\|\nabla u\|_{H^1}^2 + \|\nabla\partial_t u \|_{L^2}^2 \notag\\
&+ \|\nabla^2\tau\|_{L^2}^2 + \|u\|_{H^1}^2\Big) \mathrm{d}s \notag \\ 
\leq& C(J_0 + J_2 + J_3 + J_4 + J_5)(1+T+J_2+J_4).
\end{align}
Moreover, using \eqref{EQ-t1h}$_2$, \eqref{EQ-t1h}$_3$,  Lemma \ref{lem5.1}--\ref{lem5.3}, \eqref{2utau2}, \eqref{eq_222} and elliptic estimates, we have, for $\nabla^3(\eta,\tau)$, that
\begin{align}\label{eq_nabla3eta}
    &\int_0^t\|\nabla^3 \eta\|_{L^2}^2 \mathrm{d}s \notag\\ 
\leq&\, C\int_0^t (\|\nabla(u\cdot\nabla \eta)\|_{L^2}^2 + \|\nabla\partial_t \eta \|_{L^2}^2   + \|\eta\|_{H^1}^2)\mathrm{d}s \notag \\ 
\leq&\, C\int_0^t\Big(( \|\nabla (u,\eta)\|_{L^2}^2 + \|\partial_xu\|_{L^2}\|\partial_x\partial_yu\|_{L^2})\|\nabla (u,\eta)\|_{H^1}^2 + \|\nabla\partial_t \eta \|_{L^2}^2  + \|\eta\|_{H^1}^2\Big) \mathrm{d}s \notag \\ 
\leq& C(J_0 + J_2 + J_3 + J_4 + J_5)(1+T+J_2+J_4),
\end{align}
and 
\begin{align}\label{eq_nabla3tau}
    &\int_0^t\|\nabla^3 \tau\|_{L^2}^2 \mathrm{d}s \notag\\ 
\leq&\, C\int_0^t (\|\nabla(u\cdot\nabla \tau)\|_{L^2}^2 + \|\nabla\partial_t \tau \|_{L^2}^2 + \||\nabla^2 u ||(\tau,\eta)|\|_{L^2}^2  +  \||\nabla u ||\nabla(\tau,\eta)|\|_{L^2}^2 + \|\tau\|_{H^1}^2)\mathrm{d}s \notag \\ 
\leq&\, C\int_0^t \Big( \|\nabla (u,\eta,\tau)\|_{L^2}^2 + \|\partial_xu\|_{L^2}\|\partial_x\partial_yu\|_{L^2} + \|\nabla(\eta,\tau)\|_{H^1}^2)\|\nabla (u,\eta,\tau)\|_{H^1}^2 \notag \\ 
&+ \|\nabla\partial_t \tau \|_{L^2}^2  + \|\tau\|_{H^1}^2 \Big) \mathrm{d}s\notag \\ 
\leq& C(J_0 +J_1 + J_2 + J_3 + J_4 + J_5)(1+T+J_2+J_4+J_5),
\end{align}
Summing \eqref{eq_nablaxxu}, \eqref{eq_nabla3eta} and \eqref{eq_nabla3tau} up, we get 
\begin{align}\label{eq_uetatau_3}
    &\int^t_0 \Big(\|\nabla\partial_{x}^2u\|^{2}_{L^{2}} 
+ \|\nabla^3 \eta\|^{2}_{L^{2}} + \|\nabla^3 \tau\|^{2}_{L^{2}} \Big)\mathrm{d}s \notag\\ 
\leq\,& C(J_0 +J_1 + J_2 + J_3 + J_4 + J_5)(1+T+J_2+J_4+J_5) =: J_6.
\end{align}
The proof is complete.
\end{proof}
\subsection{Proof of Theorem \ref{thm4}}
Under the assumptions of Theorem \ref{thm4},  let $(u,\tau,\eta)$ be the solution of problem \eqref{EQ-t1h}--\eqref{OB_as_boundary_1} on $[0,T]$. 
With the help of Lemma \ref{lem5.1}--\ref{lem5.4}, one can easily find that there exist a constant $C_T$ such that 
\begin{gather}\label{eq_regu}
    \|(\eta,\tau)\|_{L^1} + \|(u,\eta,\tau)\|_{H^2}^2 + \|(\partial_tu,\partial_t\eta,\partial_t\tau)\|_{L^2}^2 + \int_0^t \|(\partial_x\partial_t u,\nabla\partial_t\eta,\nabla\partial_t \tau)\|_{L^2}^2\mathrm{d}s \notag \\ 
    + \int_0^t \|(\partial_x u,\nabla\eta,\nabla \tau)\|_{H^2}^2\mathrm{d}s \leq C_T.
\end{gather}
Then, combining \eqref{eq_regu} and Proposition \ref{prop_local}, the proof of Theorem \ref{thm4} can be completed by the standard  continuity  method. One can refer to \cite{Bedrossian_Vicol_2022} and \cite{Constantin-Kliegl2012}  for more detailed discuss.
\section{Proof of Theorem \ref{thm5} and Theorem \ref{thm6}.}\label{part3}
In this section, we will prove Theorem \ref{thm5} and Theorem \ref{thm6} for the spatial domain $[0,1]\times \mathcal{T}$. To prove Theorem \ref{thm5}, we use the  smallness condition of initial data to establish some uniform estimates of solutions. And combining with continuity method, we can complete the proof. In order to prove Theorem \ref{thm6}, we mainly use the fact that $u$ satisfying Poincar\'{e} inequality.
\subsection{The uniform estimates of \texorpdfstring{$(u,\tau,\bar{\eta})$}{u, tau, bar eta}}
\begin{lem}\label{lem6.1}
Under the assumptions of Theorem \ref{thm5}, for given $T>0$, there exist a sufficiently small positive constant $\delta_1$ independent of $T$, such that if
\begin{equation}\label{2small}
\begin{aligned}
&\sup_{0\leq s\leq T}\|(u,\tau,\bar{\eta})(s)\|_{H^{1}}\leq\delta_1,
\end{aligned}
\end{equation}
then, it holds that
\begin{equation}\label{2small2}
\begin{aligned}
&\sup_{0\leq s\leq T}\|(u,\tau,\bar{\eta})(s)\|_{H^{1}}\leq\frac{1}{2}\delta_1.
\end{aligned}
\end{equation}
\end{lem}
\begin{proof}
To begin with, multiplying $\eqref{EQ-t4h}_{1}$, $\eqref{EQ-t4h}_{2}$ and $\eqref{EQ-t4h}_{3}$ by  $2u$,  $\bar{\eta}$ and 
 $\tau$, respectively, then integrating the resulting equation over $[0,1]\times\mathcal{T}$, we have 
\begin{align}\label{0utau}
&\frac{1}{2}\frac{\mathrm{d}}{\mathrm{d}t}(\|\tau\|^{2}_{L^{2}}+2\|u\|^{2}_{L^{2}}+\|\bar{\eta}\|^{2}_{L^{2}})+2\kappa\|\tau\|^{2}_{L^{2}}
+\varepsilon\|\nabla\tau\|^{2}_{L^{2}}+2\mu\|\partial_x u\|^{2}_{L^{2}}+\varepsilon\|\nabla\bar{\eta}\|^{2}_{L^{2}} \notag\\
&\quad=\int{(\nabla u \tau+\tau\nabla^\top u):\tau}\mathrm{d}x\mathrm{d}y+\int{\bar{\eta}(\nabla u+\nabla^\top u):\tau}\mathrm{d}x\mathrm{d}y=:B_1+B_2, 
\end{align} 
where we use the fact that
\begin{equation*}
\begin{aligned}
\int{(\nabla u+\nabla^\top u):\tau}\mathrm{d}x\mathrm{d}y=-2\int{\diver\tau\cdot u}\mathrm{d}x\mathrm{d}y. 
\end{aligned}
\end{equation*}
With the help of Lemma \ref{2.3} and Young inequality, $B_1$ is bounded by
\begin{equation*}
\begin{aligned}
B_1\leq&C\|\nabla u\|_{L^{2}}\|\tau\|^{2}_{L^{4}}\\
\leq&C\|\nabla u\|_{L^{2}}\|\tau\|_{L^{2}}(\|\tau\|_{L^{2}}+\|\nabla\tau\|_{L^{2}})\\
\leq&\frac{\varepsilon}{4}\|\nabla\tau\|^2_{L^{2}}+C\|\tau\|^2_{L^{2}}(\|\nabla u\|_{L^{2}}+\|\nabla u\|^2_{L^{2}}).\\
\end{aligned}
\end{equation*}
Similarly, $B_{2}$ can be bounded by
\begin{equation*}
\begin{aligned}
B_2 =& -\int\Big(\partial_j(\bar{\eta}\tau_{ij})  u_i + \partial_i(\bar{\eta}\tau_{ij})  u_i\Big)\mathrm{d}x\mathrm{d}y\\
\leq&C\|\nabla\tau\|_{L^{2}}\|\bar{\eta}\|^{\frac{1}{2}}_{L^{2}}(\|\eta\|_{L^{2}}+\|\partial_{y}\bar{\eta}\|_{L^{2}})^{\frac{1}{2}}
\|u\|^{\frac{1}{2}}_{L^{2}}\|\partial_{x}u\|^{\frac{1}{2}}_{L^{2}}\\
&+C\|\nabla\bar{\eta}\|_{L^{2}}\|\tau\|^{\frac{1}{2}}_{L^{2}}(\|\tau\|_{L^{2}}+\|\partial_{y}\tau\|_{L^{2}})^{\frac{1}{2}}
\|u\|^{\frac{1}{2}}_{L^{2}}\|\partial_{x}u\|^{\frac{1}{2}}_{L^{2}}\\
\leq&\frac{\varepsilon}{4}\|\nabla\tau\|^2_{L^{2}}+\frac{\varepsilon}{2}\|\nabla\bar{\eta}\|^2_{L^{2}}+C\|\partial_{x}u\|^2_{L^{2}}
\Big(\|\bar{\eta}\|_{L^{2}}\big(\|\bar{\eta}\|_{L^{2}}\\
&+\|\nabla\bar{\eta}\|_{L^{2}}\big)+\|\tau\|_{L^{2}}\big(\|\tau\|_{L^{2}}+\|\nabla\tau\|_{L^{2}}\big)\Big).
\end{aligned}
\end{equation*}
Together with \eqref{2small}, inserting $B_{1}$ and $B_{2}$ into \eqref{0utau}, we get  
\begin{equation*}
\begin{aligned}
&\frac{\mathrm{d}}{\mathrm{d}t}(\|\tau\|^{2}_{L^{2}}+2\|u\|^{2}_{L^{2}}+\|\bar{\eta}\|^{2}_{L^{2}})+2\kappa\|\tau\|^{2}_{L^{2}}
+\varepsilon\|\nabla\tau\|^{2}_{L^{2}}+4\mu\|\partial_x u\|^{2}_{L^{2}}+\varepsilon\|\nabla\bar{\eta}\|^{2}_{L^{2}}\\
&\quad\leq C\|\partial_{x}u\|^2_{L^{2}}
\Big[\|\bar{\eta}\|_{L^{2}}(\|\bar{\eta}\|_{L^{2}}
+\|\nabla\bar{\eta}\|_{L^{2}})
+\|\tau\|_{L^{2}}(\|\tau\|_{L^{2}}+\|\nabla\tau\|_{L^{2}})\Big]\\
&\quad\quad +C\|\tau\|^2_{L^{2}}(\|\nabla u\|_{L^{2}}+\|\nabla u\|^2_{L^{2}})\\
&\quad\leq C\delta_1(\|\tau\|^2_{L^{2}}+\|\partial_{x}u\|^2_{L^{2}}).
\end{aligned}
\end{equation*}
By the smallness of $\delta_1$, we get
\begin{equation}\label{utau1}
\begin{aligned}
\frac{\mathrm{d}}{\mathrm{d}t}(\|\tau\|^{2}_{L^{2}}+2\|u\|^{2}_{L^{2}}+\|\bar{\eta}\|^{2}_{L^{2}})+\kappa\|\tau\|^{2}_{L^{2}}
+\varepsilon\|\nabla\tau\|^{2}_{L^{2}}+2\mu\|\partial_x u\|^{2}_{L^{2}}+\varepsilon\|\nabla\bar{\eta}\|^{2}_{L^{2}}\leq 0.
\end{aligned}
\end{equation}
Integrating \eqref{utau1} over $(0,t)$, and using Poincar\'{e} inequality, we can deduce that
\begin{align}\label{eq_uetateu1}
2\|u\|^{2}_{L^{2}} + \|\tau\|^{2}_{L^{2}}&+\|\bar{\eta}\|^{2}_{L^{2}}
+\mu\int^t_0(\| u\|^{2}_{L^{2}} + \|\partial_x u\|^{2}_{L^{2}})\mathrm{d}s+\kappa\int^t_0{\|\tau\|^{2}_{L^{2}}}\mathrm{d}s
\notag\\
&+\varepsilon\int^t_0{\|\nabla\tau\|^{2}_{L^{2}}}\mathrm{d}s
+\varepsilon\int^t_0(\|\bar{\eta}\|^{2}_{L^{2}} + \|\nabla\bar{\eta}\|^{2}_{L^{2}})\mathrm{d}s
\leq C\|(u_0,\tau_0,\eta_0)\|^{2}_{L^{2}}.
\end{align}
Next, we deduce the first-order estimation of $(u,\tau,\bar{\eta})$.
Multiplying $\eqref{EQ-t4h}_{1}$, $\eqref{EQ-t4h}_{1}$, $\eqref{EQ-t4h}_{2}$ and $\eqref{EQ-t4h}_{3}$ by $\partial_{t}u$, $-2\partial^2_{y}u$, $-\Delta\bar{\eta}$ and $-\Delta\tau$ respectively, then integrating the resulting equation over $[0,1]\times\mathcal{T}$ by parts, we obtain
\begin{align}\label{uxuy2}
&\frac{1}{2}\frac{\mathrm{d}}{\mathrm{d}t}(\|\nabla\tau\|^{2}_{L^{2}}+\mu\|\partial_x u\|^{2}_{L^{2}}+2\|\partial_y u\|^{2}_{L^{2}}
+\|\nabla\bar{\eta}\|^{2}_{L^{2}})+\|\partial_t u\|^{2}_{L^{2}}\nonumber\\
&\quad\quad\quad\quad\quad\quad+2\mu\|\partial_{x}\partial_{y} u\|^{2}_{L^{2}}
+2\kappa\|\nabla\tau\|^{2}_{L^{2}}+\varepsilon\|\nabla^2\tau\|^{2}_{L^{2}}
+\varepsilon\|\nabla^2\bar{\eta}\|^{2}_{L^{2}}\nonumber\\
&\quad=\int{\diver\tau\cdot \partial_t u}\mathrm{d}x\mathrm{d}y
-\int{u\cdot\nabla u\cdot \partial_t u}\mathrm{d}x\mathrm{d}y
-2\int{\diver\tau\cdot\partial^2_{y}u}\mathrm{d}x\mathrm{d}y\nonumber\\
&\quad\quad-2\int{u\cdot\nabla u\cdot\partial^2_{y}u}\mathrm{d}x\mathrm{d}y
+\int{u\cdot\nabla\tau:\Delta\tau}\mathrm{d}x\mathrm{d}y-\int{(\nabla u+\nabla^\top u):\Delta\tau}\mathrm{d}x\mathrm{d}y\nonumber\\
&\quad\quad-\int{((\nabla u)\tau+\tau\nabla^\top u):\Delta\tau}\mathrm{d}x\mathrm{d}y
-\int{\bar{\eta}(\nabla u+\nabla^\top u):\Delta\tau}\mathrm{d}x\mathrm{d}y\nonumber\\
&\quad\quad-\int{\nabla\bar{\eta}\cdot\nabla u\cdot\nabla\bar{\eta}}\mathrm{d}x\mathrm{d}y  =:\sum_{i=1}^{9}E'_i.
\end{align}
By H\"{o}lder  inequality and Young inequality, the estimate of $E'_1$ is given directly by
\begin{equation*}\label{}
\begin{aligned}
E'_1\leq\frac{1}{10}\|\partial_t u\|^2_{L^{2}}+C\|\nabla\tau\|^2_{L^{2}}.
\end{aligned}
\end{equation*}
Similarly to the analysis of \eqref{eq_F2}, we can get estimate $E'_2=F_{2}$ as follows

\begin{equation*}\label{E''2}
\begin{aligned}
E'_2\leq &\,\frac{2}{5}\|\partial_{t}u\|^2_{L^{2}}+\frac{\mu}{2}\|\partial_{x}\partial_{y}u\|^{2}_{L^{2}}
+C\|\partial_{x}u\|^{4}_{L^{2}}\|u\|^2_{L^{2}}+C\|\partial_{y}u\|^{2}_{L^{2}}\|u\|^2_{L^{2}}\|\partial_{x}u\|^{2}_{L^{2}}\\
\leq&\,\frac{2}{5}\|\partial_{t}u\|^2_{L^{2}} +  \frac{\mu}{2}\|\partial_{x}\partial_{y}u\|^{2}_{L^{2}} + C\delta_1\|\partial_{x} u\|^{2}_{L^{2}}.
\end{aligned}
\end{equation*}
From integration by parts, we have
\begin{equation*}\label{}
\begin{aligned}
E'_3+E'_6=&-2\int{\diver\partial_{x}\tau:\partial_{x} u}\mathrm{d}x\mathrm{d}y
\leq C\|\nabla^2\tau\|_{L^{2}}\|\partial_x u\|_{L^{2}}\\
\leq& \,\frac{\varepsilon}{8}\|\nabla^2\tau\|^2_{L^{2}}+C\|\partial_x u\|^2_{L^{2}}.
\end{aligned}
\end{equation*}
Similarly to the analysis of \eqref{F'4}, for $E''_4 = F'_4$, we have 
\begin{equation*}\label{}
\begin{aligned}
E'_{4}\leq&\frac{\mu}{4}\|\partial_{x}\partial_{y}u\|^{2}_{L^{2}}+C\delta_1\|\partial_{x}u\|^{2}_{L^{2}}.
\end{aligned}
\end{equation*}
Using Lemma \ref{2.3}, $E'_{5}$ is given by
\begin{equation*}
\begin{aligned}
E'_{5}\leq&C\|\nabla^2\tau\|_{L^{2}}\|\nabla\tau\|^{\frac{1}{2}}_{L^{2}}(\|\nabla\tau\|_{L^{2}}
+\|\nabla\partial_{y}\tau\|_{L^{2}})^{\frac{1}{2}}
\|u\|^{\frac{1}{2}}_{L^{2}}(\|u\|_{L^{2}}+\|\partial_{x}u\|_{L^{2}})^{\frac{1}{2}}\\
\leq&\frac{\varepsilon}{4}\|\nabla^2\tau\|^2_{L^{2}}+2\kappa\|\nabla\tau\|^2_{L^{2}}+C\|\nabla\tau\|^{2}_{L^{2}}
\|u\|^2_{L^{2}}\|\partial_{x}u\|^{2}_{L^{2}}\\
\leq & \frac{\varepsilon}{4}\|\nabla^2\tau\|^2_{L^{2}}+2\kappa\|\nabla\tau\|^2_{L^{2}} + C\delta_1 \|\partial_{x}u\|^{2}_{L^{2}}.
\end{aligned}
\end{equation*}
Similarly to the analysis of \eqref{F'6}, we can get
\begin{equation*}
\begin{aligned}
E'_{7}=F_6\leq&\frac{\varepsilon}{8}\|\nabla^2\tau\|^2_{L^{2}}+\frac{\mu}{8}\|\partial_{x}\partial_{y}u\|^2_{L^{2}}
+C(\|\partial_{x}u\|^2_{L^{2}}+\|\partial_{y}u\|^2_{L^{2}})(\|\tau\|^4_{L^{2}}+\|\tau\|^2_{L^{2}}\|\nabla\tau\|^2_{L^{2}})\\
\leq&\,\frac{\varepsilon}{8}\|\nabla^2\tau\|^2_{L^{2}}+\frac{\mu}{8}\|\partial_{x}\partial_{y}u\|^2_{L^{2}} + C\delta_1(\|\tau\|^2_{L^{2}}+\|\nabla\tau\|^2_{L^{2}}).\\
E'_{8}=F_7\leq&\frac{\varepsilon}{8}\|\nabla^2\tau\|^2_{L^{2}}+\frac{\mu}{8}\|\partial_{x}\partial_{y}u\|^2_{L^{2}}
+C(\|\partial_{x}u\|^2_{L^{2}}+\|\partial_{y}u\|^2_{L^{2}})(\|\bar{\eta}\|^4_{L^{2}} +\|\bar{\eta}\|^2_{L^{2}}\|\nabla\bar{\eta}\|^2_{L^{2}})\\
\leq&\,\frac{\varepsilon}{8}\|\nabla^2\tau\|^2_{L^{2}}+\frac{\mu}{8}\|\partial_{x}\partial_{y}u\|^2_{L^{2}} + C\delta_1(\|\bar{\eta}\|^2_{L^{2}}+\|\nabla\bar{\eta}\|^2_{L^{2}}).
\end{aligned}
\end{equation*}
For $E'_9$, we have
\begin{equation*}\label{}
\begin{aligned}
E'_9\leq C \|\nabla\bar{\eta}\|_{L^{2}}\|\nabla\bar{\eta}\|_{H^{1}} C\|\nabla u\|^2_{L^{2}} \leq \frac{\varepsilon}{2}\|\nabla\bar{\eta}\|^2_{H^{1}} + C\delta_1\|\nabla\bar{\eta}\|_{L^{2}}.
\end{aligned}
\end{equation*}
Combining the estimates of $E'_{1}$--$E'_{9}$ with \eqref{uxuy2}, it holds that
\begin{align} \label{eq_uetateu2}
&\frac{\mathrm{d}}{\mathrm{d}t}(\|\nabla\tau\|^{2}_{L^{2}}+\mu\|\partial_x u\|^{2}_{L^{2}}+2\|\partial_y u\|^{2}_{L^{2}}+\|\nabla\bar{\eta}\|^2_{L^{2}})+\|\partial_t u\|^{2}_{L^{2}} \notag \\
&+2\mu\|\partial_{x} \partial_{y}u\|^{2}_{L^{2}}+2\kappa\|\nabla\tau\|^{2}_{L^{2}}+\varepsilon\|\nabla^2\tau\|^{2}_{L^{2}}
+\varepsilon\|\nabla^2\bar{\eta}\|^2_{L^{2}} \notag \\
\leq &\,C\|\partial_{x}u\|^{2}_{L^{2}} + C\delta_1\|(\bar{\eta},\tau)\|_{H^1}^2.
\end{align}
Summing \eqref{utau1} and $\zeta\cdot\eqref{eq_uetateu2}$ up, and choosing $\delta_1$ and $\zeta>0$ sufficiently small, we get  
\begin{align}\label{eq_uetateu_H1}
&\frac{\mathrm{d}}{\mathrm{d}t}\Big(\|\tau\|^{2}_{L^{2}}+2\|u\|^{2}_{L^{2}}+\|\bar{\eta}\|^{2}_{L^{2}} +\zeta (\|\nabla\tau\|^{2}_{L^{2}}+\mu\|\partial_x u\|^{2}_{L^{2}}+2\|\partial_y u\|^{2}_{L^{2}}+\|\nabla\bar{\eta}\|^2_{L^{2}})\Big)\notag\\
&+\frac{1}{2}\Big(\kappa\|\tau\|^{2}_{L^{2}}
+\varepsilon\|\nabla\tau\|^{2}_{L^{2}}+2\mu\|\partial_x u\|^{2}_{L^{2}}+\varepsilon\|\nabla\bar{\eta}\|^{2}_{L^{2}}\Big)+\zeta\Big(\|\partial_t u\|^{2}_{L^{2}} +2\mu\|\partial_{x} \partial_{y}u\|^{2}_{L^{2}}\notag\\
&+2\kappa\|\nabla\tau\|^{2}_{L^{2}}+\varepsilon\|\nabla^2\tau\|^{2}_{L^{2}}
+\varepsilon\|\nabla^2\bar{\eta}\|^2_{L^{2}}\Big)  
\leq  0.
\end{align}
Integrating the above inequality with respect to $t$ over $(0,t)$, one can find a generic positive constant $C$ such that
\begin{align*}
&\|(u,\bar{\eta},\tau)\|^{2}_{H^{1}}+\int^t_0{\|(\bar{\eta},\tau)\|^{2}_{H^{2}}}\mathrm{d}s
+\int^t_0{\|\partial_t u\|^{2}_{L^{2}}}\mathrm{d}s
+\int^t_0{\|(u,\partial_x u,\partial_x \partial_y u)\|^{2}_{L^{2}}}\mathrm{d}s \\
\leq\,& C\|(u_0,\bar{\eta}_0,\tau_0)\|^{2}_{H^{1}}.
\end{align*}
Let $$C\|(u_0,\bar{\eta}_0,\tau_0)\|^{2}_{H^{1}}  =: \delta\leq \frac{1}{2}\delta_1.$$ 
The proof is complete.
\end{proof}
\subsection{The proof of Theorem \ref{thm5}.}
Together with Lemma \ref{lem6.1}, Poincar\'{e} inequality  and the method of continuity, we can complete the proof of Theorem \ref{thm5}.
Next, we will prove the decay estimates of $u$, $\bar{\eta}$ and $\tau$ and finish the proof of Theorem \ref{thm6}.
\subsection{The proof of Theorem \ref{thm6}.}
\begin{proof}
 With the help of Poincar\'{e} inequality and \eqref{eq_uetateu_H1}, one can find a positive constant $\nu$, such that  
\begin{align*}
&\frac{\mathrm{d}}{\mathrm{d}t}\Big(\|\tau\|^{2}_{L^{2}}+2\|u\|^{2}_{L^{2}}+\|\bar{\eta}\|^{2}_{L^{2}} +\zeta (\|\nabla\tau\|^{2}_{L^{2}}+\mu\|\partial_x u\|^{2}_{L^{2}}+2\|\partial_y u\|^{2}_{L^{2}}+\|\nabla\bar{\eta}\|^2_{L^{2}})\Big)\\
& +\nu \Big(\|\tau\|^{2}_{L^{2}}+2\|u\|^{2}_{L^{2}}+\|\bar{\eta}\|^{2}_{L^{2}} +\zeta (\|\nabla\tau\|^{2}_{L^{2}}+\mu\|\partial_x u\|^{2}_{L^{2}}+2\|\partial_y u\|^{2}_{L^{2}}+\|\nabla\bar{\eta}\|^2_{L^{2}})\Big) 
\leq  0,
\end{align*} 
which implies that  
\begin{equation*}
\begin{aligned}
&\|(u,\bar{\eta},\tau)\|_{H^1}\leq C\|(u_0,\bar{\eta},\tau_0)\|_{H^1}\mathrm{e}^{-\nu t},~~~\forall t\in (0,+\infty).
\end{aligned}
\end{equation*}
The proof of Theorem \ref{thm6} is complete.
\end{proof}

\section*{Acknowledgement}
The authors would like to thank Prof. Baishun Lai for valuable discussions and suggestions. The work of Y. H. Wang was partially supported by the National Natural Science Foundation of China grant 12401274 and the Natural Science Foundation of Hunan Province grant 2024JJ6302. The work of H. C. Yao was partially supported by the National Natural Science Foundation of China grant 12301273 and the Guangzhou Basic and Applied Basic Research Foundation grant 2024A04J4906.

\end{document}